\documentclass[reqno]{amsart}
\usepackage[utf8]{inputenc}
\usepackage{amscd,amsfonts,amsmath,amssymb,amsthm,bbm,dsfont,centernot,mathtools}
\usepackage{color,mathrsfs,stackrel}
\usepackage[margin=1in]{geometry}
\usepackage[colorlinks=true,linkcolor=blue,citecolor=red]{hyperref}
\usepackage{yhmath}
\usepackage{fancyhdr}
\usepackage{xy}
\usepackage{eso-pic}
\usepackage{xcolor}
\usepackage{extarrows}
\usepackage{relsize}
\usepackage{caption}
\usepackage{eurosym}
\usepackage{chngcntr}
\usepackage{tikz}

\newtheorem{corollary}{Corollary}[section]
\newtheorem{lemma}[corollary]{Lemma}
\newtheorem{proposition}[corollary]{Proposition}
\newtheorem{theorem}[corollary]{Theorem}

\theoremstyle{definition}
\newtheorem{definition}[corollary]{Definition}
\newtheorem{remark}[corollary]{Remark}

\newtheorem*{acknowledgements}{\sc Acknowledgements}
\newtheorem*{declaration}{\sc Declarations of interest}

\numberwithin{equation}{section}

\allowdisplaybreaks

\newcommand{\norm}[1]{\lVert #1 \rVert}
\newcommand{\spr}[2]{\langle #1, #2\rangle}

\newcommand{\xstararrow}[2]{\xrightharpoonup[#1]{#2\ \hspace{-0.02cm}_{*}\hspace{-0.2cm}}}

\def\Xint#1{\mathchoice
{\XXint\displaystyle\textstyle{#1}}%
{\XXint\textstyle\scriptstyle{#1}}%
{\XXint\scriptstyle\scriptscriptstyle{#1}}%
{\XXint\scriptscriptstyle\scriptscriptstyle{#1}}%
\!\int}
\def\XXint#1#2#3{{\setbox0=\hbox{$#1{#2#3}{\int}$ }
\vcenter{\hbox{$#2#3$ }}\kern-.6\wd0}}
\def\dashint{\Xint-}
\def\mint{\Xint{\rotatebox[origin][30]{$-$}}}

\def \Lip {\mathop {\rm Lip}\nolimits}

\def \div {\mathop {\rm div}\nolimits}

\def \de {\mathrm d}
\def \e {\mathrm e}
\def \R {\mathbb R}
\def \F {\mathbb F}
\def \G {\mathbb G}
\def \B {\mathbb B}
\def \C {\mathbb C}
\def \N {\mathbb N}
\def \K {\mathbb K}
\def \Cc {\mathcal C_c^1}
\def \ue {u^{\epsilon}}
\def \du {\dot{u}^{\epsilon}}

\begin{document}

\title[An existence result for the fractional Kelvin-Voigt's model]{An existence result for the fractional Kelvin-Voigt's model on time-dependent cracked domains}
\author[M. Caponi]{Maicol Caponi}
\address[Maicol Caponi]{Fakult\"at Mathematik, TU Dresden, 01062 Dresden, Germany}
\email{maicol.caponi@tu-dresden.de}
\author[F. Sapio]{Francesco Sapio}
\address[Francesco Sapio]{SISSA, via Bonomea 265, 34136 Trieste, Italy}
\email{fsapio@sissa.it}
\thanks{Preprint SISSA 27/2020/MATE}

\begin{abstract}
We prove an existence result for the fractional Kelvin-Voigt's model involving Caputo's derivative on time-dependent cracked domains. We first show the existence of a solution to a regularized version of this problem. Then, we use a compactness argument to derive that the fractional Kelvin-Voigt's model admits a solution which satisfies an energy-dissipation inequality. Finally, we prove that when the crack is not moving, the solution is unique.
\end{abstract}

\maketitle

\medskip


%---------------------------
% Introduction
%---------------------------

\section{Introduction}
This paper deals with the mathematical analysis of the dynamics of elastic damping materials in the presence of external forces and time-dependent brittle fracture. In this framework, it is important to find the behavior of the deformation when the crack evolution is known. This is the first step towards the development of a complete model of dynamic crack growth in viscoelastic materials. From a mathematical point of view, this means solving the following dynamic system
\begin{equation}\label{eq:elasticity}
\ddot u(t)-\div(\sigma(t))=f(t)\quad\text{in $\Omega\setminus\Gamma_t$, $t\in(0,T)$}.
\end{equation}
In the equation above, $\Omega\subset\R^d$ represents the reference configuration of the material, the set $\Gamma_t\subset\Omega$ models the crack at time $t$ (which is prescribed), $u(t)\colon \Omega\setminus\Gamma_t\to\R^d$ is the displacement of the deformation, $\sigma(t)$ the stress tensor, and $f(t)$ is the forcing term. 

In the classical theory of linear viscoelasticity, the constitutive stress-strain relation of the so called Kelvin-Voigt's model is given by
\begin{equation}\label{eq:local}
    \sigma(t)=\C eu(t)+\B e\dot u(t)\quad\text{in $\Omega\setminus\Gamma_t$, $t\in(0,T)$},
\end{equation}
where $\C$ and $\B$ are two positive tensors acting on the space of symmetric matrices, and $ev$ denotes the symmetric part of the gradient of a function $v$ (which is defined as $ev:=\frac{1}{2}(\nabla v+\nabla v^T)$). The local model associated to~\eqref{eq:local} has already been widely studied and we can find several existence results in the literature; we refer to~\cite{C2,Sap-ca,DM-Lar,DM-Luc,NS,T1} for existence and uniqueness results in the pure elastodynamics case ($\B=0$) and in the classic Kelvin-Voigt's one.

In recent years, materials whose constitutive equations can be described by non-local models are of increasing interest. In this context, by non-local we mean that the state of the stress at instant $t$ depends not only on that instant, but also on the previous ones (long memory). For solid viscoelastic materials, some experiments are particularly in agreement with models using fractional derivative, see for example~\cite{Dra,Fra,Sch-Met,Zhu} and the reference therein. 

In this paper, we focus on the {\it fractional} Kelvin-Voigt's model, i.e. we consider the following constitutive stress-strain relation
\begin{equation*}
\sigma(t)=\C e u(t)+\B D_t^\alpha e u(t)\quad\text{in $\Omega\setminus\Gamma_t$, $t\in(0,T)$},
\end{equation*}
where $D_t^\alpha$ denotes a fractional derivative of order $\alpha\in(0,1)$. In the literature we can find several definitions for the fractional derivative of a function $g\colon (a,b)\to \R$; here we focus on the most used ones which are {\it Riemann-Liouville's derivative} of order $\alpha$ at starting point $a$
\begin{equation*}
\prescript{RL}{a}{D}_t^\alpha g(t):=\frac{1}{{\bf \Gamma}(1-\alpha)}\frac{\de}{\de t}\int_a^t\frac{g(r)}{(t- r)^\alpha}\,\de r,
\end{equation*}
and {\it Caputo's derivative} of order $\alpha$ at starting point $a$
\begin{equation*}
\prescript{C}{a}{D}_t^\alpha g(t):=\frac{1}{{\bf \Gamma}(1-\alpha)}\int_a^t\frac{\dot g(r)}{(t- r)^\alpha}\,\de r.
\end{equation*}
We recall that $\mathbf{\Gamma}$ denotes Euler's Gamma function;
%,which is defined for a complex number $z$ with positive real part as
%\begin{equation*}
%\mathbf{\Gamma}(z)=\int_0^{\infty}t^{z-1}{\rm e}^{-t}\,\de t.
%\end{equation*}
notice that in order to define Caputo's derivative the function $g$ must be differentiable, while this is not necessary for Riemann-Liouville's derivative. Given $g\in AC([a,b])$, and $t\in(a,b)$ we have the following relation between Riemann-Liouville's and Caputo's derivative (see, e.g.,~\cite{Kilbas}):
\begin{equation}\label{eq:Cap-Riem}
\prescript{RL}{a}{D}_t^\alpha g(t)=\prescript{C}{a}{D}_t^\alpha g(t)+\frac{1}{{\bf \Gamma}(1-\alpha)}\frac{g(a)}{t^\alpha}.
\end{equation}
In particular, when $g(a)=0$, these two notions coincide. For more properties regarding these two fractional derivatives, we refer for example to~\cite{Car-Co,Mai,Po,SaKiMa} and the references therein.

In this paper we use Caputo's derivative, which means we consider the dynamic system
\begin{equation}\label{eq:elasticity2}
\ddot u(t)-\div\left(\C e u(t)+ \B\prescript{C}{0}{D}_t^\alpha e u(t)\right)=f(t)\quad\text{in $\Omega\setminus\Gamma_t$, $t\in(0,T)$}.
\end{equation}
One of the quality of this definition for the fractional derivative is that the initial conditions can be imposed in the classical sense, see for example~\cite{Mai,Po}. The choice of 0 as a starting point is due to the fact that we want to couple dynamic system~\eqref{eq:elasticity} with the initial conditions at time $t=0$.

Dealing with~\eqref{eq:elasticity2} is very difficult, since in the definition of $\prescript{C}{0}{D}_t^\alpha e u(t)$ we need that $eu$ is differentiable, which is a very strong request. Hence, we rephrase Caputo’s derivative in a more suitable way. Thanks to~\eqref{eq:Cap-Riem} for $g\in AC([0,T])$ we can write
\begin{equation}\label{eq:weakCaputo}
\prescript{C}{0}{D}_t^\alpha g(t)=\frac{1}{{\bf \Gamma}(1-\alpha)}\frac{\de}{\de t}\int_0^t\frac{1}{(t-r)^\alpha}(g(r)-g(0))\,\de r.
\end{equation}
This formulation of Caputo's derivative is well-posed in the distributional sense also when the function $g$ is only integrable. We point out that formula~\eqref{eq:weakCaputo} can be found in the recent literature on fractional derivatives, where it is used to define the notion of weak Caputo's derivative for less regular functions, see for example~\cite{Val-DiP-Ve,Li-Liu}. 

Thanks to formula \eqref{eq:weakCaputo}, we can write system~\eqref{eq:elasticity2} in a weaker form (see Definition \ref{gen-sol}) as
\begin{equation}\label{eq:elasticity3}
\ddot u(t)-\div\left(\C eu(t)+\frac{\de}{\de t}\int_0^t\F(t-r)(eu(r)-eu(0))\,\de r \right)=f(t)\quad\text{in $\Omega\setminus\Gamma_t$, $t\in(0,T)$},
\end{equation}
where
\begin{equation}\label{eq:Fdef}
    \F(t):=\rho(t)\B,\quad \rho(t):=\frac{1}{{\bf \Gamma}(1-\alpha)}\frac{1}{t^\alpha}\quad\text{for $t\in(0,\infty)$}.
\end{equation}
Notice that the scalar function $\rho$ appearing in $\F$ is positive, decreasing, and convex on $(0,\infty)$. Moreover, $\rho\in L^1(0,T)$ for every $T>0$, but it is not bounded on $(0,T)$. In particular, we can not compute the derivative in front of the convolution integral in~\eqref{eq:elasticity3}. 

When there is no crack, existence results for this kind of system can be found for example in~\cite{At-Opa,CaCaVa,Opa,Opa-Su}. However, in the case of a dynamic fracture, the techniques used in the previous papers can not be applied and up to now there are no existence results in this setting. 

To prove the existence of a solution to~\eqref{eq:elasticity3} we proceed into two steps. First we consider a regularized version of~\eqref{eq:elasticity3}, where we replace the kernel $\F$ in~\eqref{eq:elasticity3} by a regular kernel $\G\in C^2([0,T])$. Then we prove the existence of a solution to the more regular system
\begin{equation}\label{eq:elasticity4}
\ddot u(t)-\div\left(\C eu(t)+\frac{\de}{\de t}\int_0^t\G(t-r)(eu(r)-eu(0))\,\de r \right)=f(t)\quad\text{in $\Omega\setminus\Gamma_t$, $t\in(0,T)$},
\end{equation}
and we show that this solution satisfies a uniform bound depending on the $L^1$-norm of $\G$. Finally, we consider a sequence of regular tensors $\G^\epsilon$ converging to $\F$ in $L^1$ and we take the solutions to~\eqref{eq:elasticity4} with $\G:=\G^\epsilon$. By a compactness argument, we show that the sequence $u^\epsilon$ converge to a function $u^*$ which solves~\eqref{eq:elasticity3}. Moreover, we prove that this solution satisfies an energy-dissipation inequality. We conclude this paper by showing that, when the crack is not moving, the fractional Kelvin-Voigt's system~\eqref{eq:elasticity3} admits a unique solution.

The paper is organized as follows: in Section~\ref{sec:not} we fix the notation and the framework of our problem. Moreover, we give the notion of solution to the fractional Kelvin-Voigt's system involving Caputo's derivative~\eqref{eq:elasticity3} and we state our main existence result (see Theorem \ref{thm:irr_exis}). Section~\ref{sec:reg} deals with the regularized system~\eqref{eq:elasticity4}. First, by a time-discretization procedure in Theorem~\ref{thm:reg_exis} we prove the existence of a solution to~\eqref{eq:elasticity4}. Then, in Lemma~\ref{lem:irr_est} we derive the uniform energy estimate which depends on the $L^1$-norm of $\G$. In Section~\ref{sec:irr} we consider Kelvin-Voigt's system~\eqref{eq:elasticity3}: we prove the existence of a generalized solution to system~\eqref{eq:elasticity3} and in Theorem~\ref{thm:irr_enin} we show that such a solution satisfies an energy-dissipation inequality. Finally, in Section~\ref{sec:uniq} we prove that, for a not moving crack, the solution to~\eqref{eq:elasticity3} is unique.

%---------------------------
% Notation and framework
%---------------------------

\section{Notation and framework of the problem}\label{sec:not}

The space of $m\times d$ matrices with real entries is denoted by $\R^{m\times d}$; in case $m=d$, the subspace of symmetric matrices is denoted by $\mathbb R^{d\times d}_{sym}$. Given a function $u\colon\R^d\to\R^m$, we denote its Jacobian matrix by $\nabla u$, whose components are $(\nabla u)_{ij}:= \partial_j u_i$ for $i=1,\dots,m$ and $j=1,\dots,d$; when $u\colon \R^d\to\R^d$, we use $eu$ to denote the symmetric part of the gradient, namely $eu:=\frac{1}{2}(\nabla u+\nabla u^T)$. Given a tensor field $A\colon \R^d\to\R^{m\times d}$, by $\div A$ we mean its divergence with respect to rows, namely $(\div A)_i:= \sum_{j=1}^d\partial_jA_{ij}$ for $i=1,\dots,m$. 

We denote the $d$-dimensional Lebesgue measure by $\mathcal L^d$ and the $(d-1)$-dimensional Hausdorff measure by $\mathcal H^{d-1}$; given a bounded open set $\Omega$ with Lipschitz boundary, by $\nu$ we mean the outer unit normal vector to $\partial\Omega$, which is defined $\mathcal H^{d-1}$-a.e. on the boundary. The Lebesgue and Sobolev spaces on $\Omega$ are defined as usual; the boundary values of a Sobolev function are always intended in the sense of traces. 

The norm of a generic Banach space $X$ is denoted by $\|\cdot\|_X$; when $X$ is a Hilbert space, we use $(\cdot,\cdot)_X$ to denote its scalar product. We denote by $X'$ the dual of $X$ and by $\spr{\cdot}{\cdot}_{X'}$ the duality product between $X'$ and $X$. Given two Banach spaces $X_1$ and $X_2$, the space of linear and continuous maps from $X_1$ to $X_2$ is denoted by $\mathscr L(X_1;X_2)$; given $\mathbb A\in\mathscr L(X_1;X_2)$ and $u\in X_1$, we write $\mathbb A u\in X_2$ to denote the image of $u$ under $\mathbb A$. 

Moreover, given an open interval $(a,b)\subseteq\R$ and $p\in[1,\infty]$, we denote by $L^p(a,b;X)$ the space of $L^p$ functions from $(a,b)$ to $X$; we use $W^{k,p}(a,b;X)$ and $H^k(a,b;X)$ (for $p=2$) to denote the Sobolev space of functions from $(a,b)$ to $X$ with $k$ derivatives. Given $u\in W^{1,p}(a,b;X)$, we denote by $\dot u\in L^p(a,b;X)$ its derivative in the sense of distributions. When dealing with an element $u\in W^{1,p}(a,b;X)$ we always assume $u$ to be the continuous representative of its class; in particular, it makes sense to consider the pointwise value $u(t)$ for every $t\in[a,b]$. We use $C_w^0([a,b];X)$ to denote the set of weakly continuous functions from $[a,b]$ to $X$, namely, the collection of maps $u\colon [a,b]\to X$ such that $t\mapsto \spr{x'}{u(t)}_{X'}$ is continuous from $[a,b]$ to $\R$ for every $x'\in X'$.

Let $T$ be a positive real number and let $\Omega\subset\R^d$ be a bounded open set with Lipschitz boundary. Let $\partial_D\Omega$ be a (possibly empty) Borel subset of $\partial\Omega$ and let $\partial_N\Omega$ be its complement. Throughout the paper we assume the following hypotheses on the geometry of the cracks:
\begin{itemize}
\item[(H1)] $\Gamma\subset\overline\Omega$ is a closed set with $\mathcal L^d(\Gamma)=0$ and $\mathcal H^{d-1}(\Gamma\cap\partial\Omega)=0$;
\item[(H2)] for every $x\in\Gamma$ there exists an open neighborhood $U$ of $x$ in $\R^d$ such that $(U\cap\Omega)\setminus\Gamma$ is the union of two disjoint open sets $U^+$ and $U^-$ with Lipschitz boundary;
\item[(H3)] $\{\Gamma_t\}_{t\in[0,T]}$ is an increasing family in time of closed subsets of $\Gamma$, i.e. $\Gamma_s\subset\Gamma_t$ for every $0\le s\le t\le T$.
\end{itemize}

Thanks (H1)--(H3) the space $L^2(\Omega\setminus\Gamma_t;\R^m)$ coincides with $L^2(\Omega;\R^m)$ for every $t\in[0,T]$ and $m\in\mathbb N$. In particular, we can extend a function $u\in L^2(\Omega\setminus\Gamma_t;\R^m)$ to a function in $L^2(\Omega;\R^m)$ by setting $u=0$ on $\Gamma_t$. To simplify our exposition, for every $m\in\mathbb N$ we define the spaces $H:=L^2(\Omega;\R^m)$, $H_N:=L^2(\partial_N\Omega;\R^m)$ and $H_D:=L^2(\partial_D\Omega;\R^m)$; we always identify the dual of $H$ by $H$ itself, and $L^2((0,T)\times\Omega;\R^m)$ by the space $L^2(0,T;H)$. We define
\begin{equation*}
U_t:= H^1(\Omega\setminus\Gamma_t;\R^d)\quad\text{for every $t\in[0,T]$}.
\end{equation*}
Notice that in the definition of $U_t$ we are considering only the distributional gradient of $u$ in $\Omega\setminus\Gamma_t$ and not the one in $\Omega$. By (H2) we can find a finite number of open sets $U_j\subset\Omega\setminus\Gamma$, $j=1,\dots m$, with Lipschitz boundary, such that $\Omega\setminus\Gamma=\cup_{j=1}^m U_j$. By using second Korn's inequality in each $U_j$ (see, e.g.,~\cite[Theorem~2.4]{OSY}) and taking the sum over $j$ we can find a constant $C_K$, depending only on $\Omega$ and $\Gamma$, such that 
\begin{equation*}
\norm{\nabla u}_H^2\le C_K\left(\norm{u}_H^2+\norm{e u}_H^2\right)\quad\text{for every }u\in H^1(\Omega\setminus\Gamma;\R^d),
\end{equation*}
where $eu$ is the symmetric part of $\nabla u$. Therefore, we can use on the space $U_t$ the equivalent norm
\begin{equation*}
\norm{u}_{U_t}:=(\norm{u}_{H}^2+\norm{e u}_{H}^2)^{\frac{1}{2}}\quad\text{for every }u\in U_t.
\end{equation*}
Furthermore, the trace of $u\in H^1(\Omega\setminus\Gamma;\R^d)$ is well defined on $\partial\Omega$. Indeed, we may find a finite number of open sets with Lipschitz boundary $V_k\subset\Omega\setminus\Gamma$, $k=1,\dots l$, such that $\partial\Omega\setminus(\Gamma\cap\partial\Omega)\subset\cup_{k=1}^l\partial V_k$. Since $\mathcal H^{d-1}(\Gamma\cap\partial\Omega)=0$, there exists a constant $C$, depending only on $\Omega$ and $\Gamma$, such that
\begin{equation*}
\norm{u}_{L^2(\partial\Omega;\R^d)}\le C\norm{u}_{H^1(\Omega\setminus\Gamma;\R^d)}\quad\text{for every }u\in H^1(\Omega\setminus\Gamma;\R^d).
\end{equation*}
Hence, we can consider the set
\begin{equation*}
U_t^D:=\{u\in U_t:u=0\text{ on }\partial_D\Omega\}\quad\text{for every $t\in[0,T]$},
\end{equation*}
which is a closed subspace of $U_t$. Moreover, there exists a positive constant $C_{tr}$ such that 
\begin{equation*}
\norm{u}_{H_N}\leq C_{tr}\norm{u}_{U_T}\quad\text{for every }u\in U_T.
\end{equation*}
Now, we define the following sets of functions
\begin{align*}
&\mathcal C_w:=\{u\in C_w^0([0,T];U_T):\text{$\dot u\in C_w^0([0,T];H)$, $u(t)\in U_t$ for every $t\in[0,T]$}\},\\
&\mathcal C^1_c:=\{\varphi\in C^1_c(0,T;U^D_T):\text{$\varphi(t)\in U^D_t$ for every $t\in[0,T]$}\},
\end{align*}
in which we develop our theory. Moreover, we consider the Banach space
\begin{align*}
B:=L^{\infty}(\Omega;\mathcal{L}_{sym}(\R^{d\times d}_{sym},\R^{d\times d}_{sym})),
\end{align*}
where $\mathcal{L}_{sym}(\R^{d\times d}_{sym},\R^{d\times d}_{sym})$ represents the space of symmetric tensor fields, i.e. the collections of linear and continuous maps $\mathbb A\colon \R^{d\times d}_{sym}\to \R^{d\times d}_{sym}$ satisfying
\begin{equation*}
 \mathbb A\xi\cdot\eta=\mathbb A\eta\cdot\xi\quad\text{for every $\xi,\eta\in\R^{d\times d}_{sym}$}.
\end{equation*}

We assume that the Dirichlet datum $z$, the Neumann datum $N$, the forcing term $f$, the initial displacement $u^0$, and the initial velocity $u^1$ satisfy
\begin{align}
&z\in W^{2,1}(0,T;U_0),\label{eq:dataFz}\\
&N\in W^{1,1}(0,T;H_N),\quad f\in L^2(0,T;H),\label{eq:data2}\\
&u^0\in U_0\text{ with }u^0-z(0)\in U_0^D,\quad u^1\in H.
\end{align}
We consider a coercive tensor $\C\in B$, which means that there exists $\gamma>0$ such that
\begin{align}\label{eq:C}
\C (x)\xi\cdot \xi\ge \gamma|\xi|^2\quad\text{for every $\xi\in \R^d$ and a.e. $x\in\Omega$}.
\end{align}
Moreover, let us take a time-dependent tensor $\F\colon (0,T+\delta_0)\to B$, with $\delta_0>0$, satisfying 
\begin{align}
&\F\in C^2(0,T+\delta_0;B)\cap L^1(0,T+\delta_0;B),\label{eq:F1}\\
&\F(t,x)\xi\cdot \xi\ge 0&&\text{for every $\xi\in \R^d$, $t\in(0,T+\delta_0)$, and a.e. $x\in\Omega$},\label{eq:F2}\\
&\dot\F(t,x)\xi\cdot \xi\le 0 &&\text{for every $\xi\in \R^d$, $t\in(0,T+\delta_0)$, and a.e. $x\in\Omega$},\label{eq:F3}\\
&\ddot\F(t,x)\xi\cdot \xi\ge 0 &&\text{for every $\xi\in \R^d$, $t\in(0,T+\delta_0)$, and a.e. $x\in\Omega$}\label{eq:F4}.
\end{align}

\begin{remark}
The tensor $\F$ may be not defined at $t=0$ and unbounded on $(0,T+\delta_0)$. In the case of~\eqref{eq:Fdef}, the function $\F$ associated to the fractional Kelvin-Voigt's model involving Caputo's derivative, satisfies~\eqref{eq:F1}--\eqref{eq:F4} provided that $\B\in B$ is non-negative, that is
\begin{equation*}
\B (x)\xi\cdot \xi\ge 0\quad\text{for every $\xi\in \R^d$ and a.e. $x\in\Omega$}.
\end{equation*}
Since in our existence result we first regularize the tensor $\F$ by means of translations (see Section \ref{sec:irr}) we need that $\F$ is defined also on the right of $T$. This is not a problem, because our standard example for $\F$, which is~\eqref{eq:Fdef}, is defined on the whole $(0,\infty)$.
\end{remark}

In this paper we want to study the following problem
\begin{equation}\label{eq:Fsystem}
\begin{cases}
\ddot u(t)-\div(\C eu(t))-\div\left(\frac{\de}{\de t}\int_0^t\F(t-r)(eu(r)-eu^0)\,\de r\right)=f(t)&\text{in $\Omega\setminus \Gamma_t$,\hspace{0.2cm} $t\in(0,T)$},\\
u(t)=z(t)&\text{on $\partial_D\Omega$,\hspace{0.35cm} $t\in(0,T)$},\\
\C eu(t)\nu+\left(\frac{\de}{\de t}\int_0^t\F(t-r)(e u(r)-eu^0)\,\de r\right)\nu =N(t)&\text{on $\partial_N\Omega$,\hspace{0.33cm} $t\in(0,T)$},\\
\C eu(t)\nu+\left(\frac{\de}{\de t}\int_0^t\F(t-r)(e u(r)-eu^0)\,\de r\right)\nu =0&\text{on $\Gamma_t$,\hspace{0.68cm} $t\in(0,T)$},\\
u(0)=u^0,\quad\dot u(0)=u^1&\text{in $\Omega\setminus\Gamma_0$}.
\end{cases}
\end{equation}

We give the following notion of solution to system~\eqref{eq:Fsystem}:

\begin{definition}[Generalized solution]\label{gen-sol}
Assume~\eqref{eq:dataFz}--\eqref{eq:F4}. A function $u\in \mathcal C_w$ is a {\it generalized solution} to system~\eqref{eq:Fsystem} if $u(t)-z(t)\in U_t^D$ for every $t\in[0,T]$, $u(0)=u^0$ in $U_0$, $\dot u(0)=u^1$ in $H$, and for every $\varphi\in \Cc $ the following equality holds
\begin{align}\label{eq:Fgen}
&-\int_0^T(\dot u(t),\dot \varphi(t))_H\,\de t+\int_0^T(\C eu(t),e\varphi(t))_H\,\de t-\int_0^T\int_0^t(\F(t-r)(e u(r)-eu^0),e \dot \varphi(t))_{H}\,\de r\,\de t\nonumber\\
&=\int_0^T(f(t),\varphi(t))_H\,\de t+\int_0^T(N(t),\varphi(t))_{H_N}\,\de t.
\end{align}
\end{definition}

\begin{remark}
The Neumann conditions appearing in~\eqref{eq:Fsystem} are only formal; they are used to pass from the strong formulation in~\eqref{eq:Fsystem} to the weak one~\eqref{eq:Fgen}.
\end{remark}

The main existence result of this paper is the following theorem:

\begin{theorem}\label{thm:irr_exis}
Assume~\eqref{eq:dataFz}--\eqref{eq:F4}. Then there exists a generalized solution $u\in \mathcal C_w$ to system~\eqref{eq:Fsystem}.
\end{theorem}

The proof of this theorem requires several preliminary results. First, in the next section, we prove the existence of a generalized solution when the tensor $\F$ is replaced by a tensor $\G\in C^2([0,T];B)$. Then, we show that such a solution satisfies an energy estimate, which depends via $\G$ only by its $L^1$-norm. In Section~\ref{sec:irr} we combine these two results to prove Theorem~\ref{thm:irr_exis}.

%---------------------------
% The regularized model
%---------------------------

\section{The regularized model}\label{sec:reg}

In this section we deal with a regularized version of the system~\eqref{eq:Fsystem}, where the tensor $\F$ is replaced by a tensor $\G$ which is bounded at $t=0$. More precisely, we consider the following system
\begin{equation}\label{eq:Gsystem}
\begin{cases}
\ddot u(t)-\div(\C eu(t))-\div\left(\frac{\de}{\de t}\int_0^t\G(t-r)(eu(r)-eu^0)\,\de r\right)=f(t)&\text{in $\Omega\setminus \Gamma_t$,\hspace{0.2cm} $t\in(0,T)$},\\
u(t)=z(t)&\text{on $\partial_D\Omega$,\hspace{0.35cm} $t\in(0,T)$},\\
\C eu(t)\nu+\left(\frac{\de}{\de t}\int_0^t\G(t-r)(e u(r)-eu^0)\,\de r\right)\nu =N(t)&\text{on $\partial_N\Omega$,\hspace{0.33cm} $t\in(0,T)$},\\
\C eu(t)\nu+\left(\frac{\de}{\de t}\int_0^t\G(t-r)(e u(r)-eu^0)\,\de r\right)\nu =0&\text{on $\Gamma_t$,\hspace{0.68cm} $t\in(0,T)$},\\
u(0)=u^0,\quad\dot u(0)=u^1&\text{in $\Omega\setminus\Gamma_0$},
\end{cases}
\end{equation}
and we assume that $\G\colon [0,T]\to B$ satisfies
\begin{align}
&\G\in C^2([0,T];B)\label{eq:G1},\\
&\G(t,x)\xi\cdot \xi\ge 0&&\text{for every $\xi\in \R^d$, $t\in[0,T]$, and a.e. $x\in\Omega$},\label{eq:G2}\\
&\dot\G(t,x)\xi\cdot \xi\le 0 &&\text{for every $\xi\in \R^d$, $t\in[0,T]$, and a.e. $x\in\Omega$},\label{eq:G3}\\
&\ddot\G(t,x)\xi\cdot \xi\ge 0 &&\text{for every $\xi\in \R^d$, $t\in[0,T]$, and a.e. $x\in\Omega$}\label{eq:G4}.
\end{align}
As before, on $N$, $u^0$, $u^1$, and $\C$ we assume~\eqref{eq:data2}--\eqref{eq:C}, while for the Dirichlet datum $z$ we can require the weaker assumption
\begin{align}
&z\in W^{2,1}(0,T;H)\cap W^{1,1}(0,T;U_0).\label{eq:data1}
\end{align}

The notion of generalized solution to~\eqref{eq:Gsystem} is the same as before.

\begin{definition}[Generalized solution]
Assume~\eqref{eq:data2}--\eqref{eq:C} and~\eqref{eq:G1}--\eqref{eq:data1}. A function $u\in \mathcal C_w$ is a {\it generalized solution} to system~\eqref{eq:Gsystem} if $u(t)-z(t)\in U_t^D$ for every $t\in[0,T]$, $u(0)=u^0$ in $U_0$, $\dot u(0)=u^1$ in $H$, and for every $\varphi\in \Cc $ the following equality holds
\begin{align}\label{eq:Ggen}
&-\int_0^T(\dot u(t),\dot \varphi(t))_H\,\de t+\int_0^T(\C eu(t),e\varphi(t))_H\,\de t-\int_0^T\int_0^t(\G(t-r)(e u(r)-eu^0),e \dot \varphi(t))_{H}\,\de r\,\de t\nonumber\\
&=\int_0^T(f(t),\varphi(t))_H\,\de t+\int_0^T(N(t),\varphi(t))_{H_N}\,\de t.
\end{align}
\end{definition}

Since the time-dependent tensor $\G$ is well defined in $t=0$, we can give another notion of solution. In particular, the convolution integral is now differentiable, and we can write 
\begin{equation*}
\frac{\de}{\de t}\int_0^t\G(t-r)(eu(r)-eu^0)\,\de r=\G(0) (eu(t)-eu^0)+\int_0^t\dot\G(t-r)(eu(r)-eu^0)\,\de r.
\end{equation*}

\begin{definition}[Weak solution]
Assume~\eqref{eq:data2}--\eqref{eq:C} and~\eqref{eq:G1}--\eqref{eq:data1}. A function $u\in \mathcal C_w$ is a {\it weak solution} to system~\eqref{eq:Gsystem} if $u(t)-z(t)\in U_t^D$ for every $t\in[0,T]$, $u(0)=u^0$ in $U_0$, $\dot u(0)=u^1$ in $H$, and for every $\varphi\in\Cc$ the following equality holds
\begin{align}\label{eq:weaksol}
&-\int_0^T(\dot u(t),\dot\varphi(t))_H\,\de t+\int_0^T(\C eu(t),e\varphi(t))_H\,\de t+\int_0^T(\G(0) (eu(t)-eu^0), e \varphi(t))_{H}\,\de t\nonumber\\
&+\int_0^T\int_0^t(\dot\G(t-r)(eu(r)-eu^0) , e \varphi(t))_{H}\,\de r\,\de t=\int_0^T(f(t),\varphi(t))_H\,\de t+\int_0^T(N(t),\varphi(t))_{H_N}\,\de t.
\end{align}
\end{definition}
In this framework the two previous definitions are equivalent.
\begin{proposition}\label{lem:equiv}
Assume~\eqref{eq:data2}--\eqref{eq:C} and~\eqref{eq:G1}--\eqref{eq:data1}. Then $u\in\mathcal C_w$ is a generalized solution to~\eqref{eq:Gsystem} if and only if $u$ is a weak solution.
\end{proposition}

\begin{proof}
We only need to prove that~\eqref{eq:weaksol} is equivalent to~\eqref{eq:Ggen}. This is true if and only if the function $u\in \mathcal C_w$ satisfies for every $\varphi\in \Cc$ the following equality
\begin{align}\label{eq:genweak}
\int_0^T(\G(0)(eu(t)-eu^0), e \varphi(t))_{H}\,\de t&+\int_0^T\int_0^t(\dot\G(t-r)(eu(r)-eu^0), e \varphi(t))_{H}\,\de r\,\de t\nonumber\\
&=-\int_0^T\int_0^t(\G(t-r)(eu(r)-eu^0), e \dot\varphi(t))_{H}\,\de r\,\de t.
\end{align}
Let us consider for $t\in[0,T]$ the function 
\begin{equation*}
p(t):=\int_0^t(\G(t-r)(eu(r)-eu^0), e \varphi(t))_{H}\,\de r.
\end{equation*}
We claim that $p\in \Lip([0,T])$. Indeed, for every $s,t\in[0,T]$ with $s<t$ we have
\begin{align*}
|p(s)-p(t)|&\le\left|\int_s^t(\G(t- r) (eu(r)-eu^0), e \varphi(t))_{H}\,\de r \right|+\left|\int_0^s(\G(s- r)(e u(r)-eu^0), e \varphi(t)- e \varphi(s))_{H}\,\de r \right|\\
&\quad+ \left|\int_0^s((\G(t- r)-\G(s- r)) (eu(r)-eu^0), e \varphi(t))_{H}\,\de r \right|.
\end{align*}
Since
\begin{align*}
&\left|\int_s^t(\G(t- r) (eu(r)-eu^0), e \varphi(t))_{H}\,\de r \right|\le 2(t-s)\|\G\|_{C^0([0,T];B)}\| e \varphi\|_{C^0([0,T];H)}\| e u\|_{L^\infty(0,T;H)},\\
&\left|\int_0^s(\G(s- r) (eu(r)-eu^0), e \varphi(t)- e \varphi(s))_{H}\,\de r \right|\le 2(t-s)\|\G\|_{C^0([0,T];B)}\| e \dot\varphi\|_{C^0([0,T];H)}T\| e u\|_{L^\infty(0,T;H)},\\
&\left|\int_0^s((\G(t- r)-\G(s- r)) (eu(r)-eu^0), e \varphi(t))_{H}\,\de r \right|\le 2(t-s)\|\dot\G\|_{C^0([0,T];B)}\| e \varphi\|_{C^0([0,T];H)}T\| e u\|_{L^\infty(0,T;H)},
\end{align*}
we deduce that $p\in \Lip([0,T])$. In particular, there exists $\dot p(t)$ for a.e. $t\in(0,T)$. Given $t\in(0,T)$ and $h>0$ we can write
\begin{align*}
&\frac{p(t+h)-p(t)}{h}=\int_0^t(\frac{\G(t+h- r)-\G(t- r)}{h} (e u(r)-eu^0), e \varphi(t+h))_{H}\,\de r \\
&\quad+\dashint_t^{t+h}(\G(t+h- r) (e u(r)-eu^0), e \varphi(t+h))_{H}\,\de r+\int_0^t(\G(t- r) (e u(r)-eu^0),\frac{e\varphi(t+h)- e\varphi(t)}{h})_{H}\,\de r.
\end{align*}
Let us compute these three limits separately. We claim that for a.e. $t\in(0,T)$ we have
\begin{equation*}
\lim_{h\to 0^+}\dashint_t^{t+h}(\G(t+h- r) (e u(r)-eu^0), e \varphi(t+h))_{H}\,\de r =(\G(0) (e u(t)-eu^0), e \varphi(t))_{H}.
\end{equation*}
Indeed, by the Lebesgue's differentiation theorem, for a.e. $t\in(0,T)$ we get
\begin{align*}
&\left|\dashint_t^{t+h}(\G(t+h- r) (e u(r)-eu^0), e \varphi(t+h))_{H}\,\de r -(\G(0) (eu(t)-eu^0), e \varphi(t))_{H}\right|\\
&\le \|\G(0)\|_B\| e \varphi(t)\|_{H}\dashint_t^{t+h}\| eu(t)-eu(r)\|_{H}\,\de r+\|\G(0)\|_B\| e \varphi(t+h)- e \varphi(t)\|_{H}\dashint_t^{t+h}\| e u(r)-eu^0\|_{H}\,\de r \\
&\quad+\| e \varphi(t+h)\|_{H}\dashint_t^{t+h}\|\G(t+h- r)-\G(0)\|_B\| e u(r)-eu^0\|_{H}\,\de r\xrightarrow[h\to 0^+]{} 0.
\end{align*}
Moreover, for every $t\in(0,T)$ we have
\begin{equation*}
\lim_{h\to 0^+}\int_0^t(\frac{\G(t+h- r)-\G(t- r)}{h}(e u(r)-eu^0), e \varphi(t+h))_{H}\,\de r =\int_0^t(\dot\G(t- r) (e u(r)-eu^0), e \varphi(t))_{H}\,\de r 
\end{equation*}
since
\begin{align*}
e\varphi(t+h)\xrightarrow[h\to 0^+]{H} e \varphi(t),\quad\frac{\G(t+h-\,\cdot\,)-\G(t-\,\cdot\,)}{h} (eu(\cdot)-eu^0)\xrightarrow[h\to 0^+]{L^1(0,t;H)} \dot\G(t-\,\cdot\,) (eu(\cdot)-eu^0).
\end{align*}
Finally, for every $t\in(0,T)$ we get
\begin{equation*}
\lim_{h\to 0^+}\int_0^t(\G(t- r) (e u(r)-eu^0),\frac{ e \varphi(t+h)- e \varphi(t)}{h})_{H}\,\de r =\int_0^t(\G(t- r) (e u(r)-eu^0), e \dot\varphi(t))_{H}\,\de r 
\end{equation*}
because 
$$\frac{ e \varphi(t+h)- e \varphi(t)}{h}\xrightarrow[h\to 0^+]{H} e \dot\varphi(t).$$
Therefore, by the identity
\begin{equation*}
 0=p(T)-p(0)=\int_0^T\dot p(t)\,\de t
\end{equation*}
and the previous computations we deduce~\eqref{eq:genweak}.
\end{proof}

In the particular case in which the tensor $\G$ appearing in~\eqref{eq:Gsystem} is the one associated to the Standard viscoelastic model, i.e.
\begin{equation*}
    \G(t)=\frac{1}{\beta}\e^{-\frac{t}{\beta}}\B\quad\text{for $t\in[0,T]$}
\end{equation*}
with $\beta>0$ and $\B\in B$ non-negative tensor, then the existence of weak solutions (and so generalized solutions) was proved in~\cite{Sap}. Here we adapt the techniques of~\cite{Sap} to a general tensor $\G$ satisfying~\eqref{eq:G1}--\eqref{eq:G4}.

\subsection{Existence and energy-dissipation inequality}

In this subsection we prove the existence of a generalized solution to system~\eqref{eq:Gsystem}, by means of a time discretization scheme in the same spirit of~\cite{DM-Lar}. Moreover, we show that such a solution satisfies the energy-dissipation inequality~\eqref{eq:enin}. 

We fix $n\in\mathbb N$ and we set
\begin{equation*}
 \tau_n:=\frac{T}{n},\quad u_n^0:=u^0,\quad u_n^{-1}:=u^0-\tau_nu^1,\quad\delta z_n^0:=\dot z(0),\quad \delta \G_n^0:=0.
\end{equation*}
Let us define
\begin{align*}
&U_n^j:=U^D_{j\tau_n},& & z_n^j:=z(j\tau_n),& &  \G_n^j:=\G(j\tau_n)& & & &\text{for $j=0,\dots,n$},\\
&\delta z_n^j:=\frac{z_n^j- z_n^{j-1}}{\tau_n}, & & \delta^2 z_n^j:=\frac{\delta z_n^j-\delta z_n^{j-1}}{\tau_n},& & \delta \G_n^j:=\frac{\G_n^j-\G_n^{j-1}}{\tau_n},& &\delta ^2\G_n^j:=\frac{\delta \G_n^j-\delta \G_n^{j-1}}{\tau_n}& &\text{for $j=1,\dots,n$}.
\end{align*}
Regarding the forcing term and the Neumann datum we pose
\begin{align*}
&N^j_n:=N(j\tau_n)& &\text{for }j=0,\dots,n,\\ &f_n^j:=\dashint_{(j-1)\tau_n}^{j\tau_n} f(r)\,\de r,\quad \delta N_n^j:=\frac{N_n^j-N_n^{j-1}}{\tau_n}& &\text{for }j=1,\dots,n.
\end{align*}
For every $j=1,\dots,n$ let us consider the unique $u_n^j\in U_T$ with $u^j_n-z_n^j\in U_n^j$, which satisfies
\begin{equation}\label{eq:un}
(\delta^2 u_n^j,v)_H+(\C eu_n^j,ev)_H+(\G_n^0(e u_n^j-eu^0),e v)_{H}+\sum_{k=1}^j\tau_n(\delta \G_n^{j-k}(e u_n^k-eu^0),e v)_{H}=(f_n^j,v)_H+(N_n^j,v)_{H_N}
\end{equation}
for every $v\in U_n^j$, where
\begin{align*}
\delta u_n^j:=\frac{u_n^j-u_n^{j-1}}{\tau_n}\quad\text{for $j=0,\dots,n$},\quad \delta ^2u_n^j:=\frac{\delta u_n^j-\delta u_n^{j-1}}{\tau_n}\quad\text{for $j=1,\dots,n$}.
\end{align*}
The existence and uniqueness of $u_n^j$ is a consequence of Lax-Milgram's lemma. Notice that equation~\eqref{eq:un} is a sort of discrete version of~\eqref{eq:weaksol}, which we already know that is equivalent to~\eqref{eq:Ggen}.
%Indeed, by considering $n\in \N$ and the bilinear form $B_n\colon U_T\times U_T\to \R$ defined as
% \begin{equation*}
% B_n(u,v):=\frac{1}{\tau_n^2}(u,v)_H+(\G_n^0e u,e v)_{H}\quad\text{for $u,v\in U_T$},
% \end{equation*}
% and for every $j=1,\dots,n$ the functional $g_n^j\in U'_T$ defined as
% \begin{equation*}
% \langle g_n^j,v\rangle:=\frac{1}{\tau_n^2}(2u_n^{j-1}-u_n^{j-2},v)_H-\sum_{k=1}^j\tau_n(\delta \G_n^{j-k}e u_n^k,e v)_{H}+(f_n^j,v)_H\quad\text{for $v\in U_T$},
% \end{equation*}
% then problem~\eqref{eq:un} can be rephrased as: for $j=1,\dots,n$ find $u_n^j\in U_n^j$, with $u_n^j-z_n^j\in U_n^j$, such that
% \begin{equation*}
% B_n(u_n^j,v)=\langle g_n^j,v\rangle_{U_T'}\quad\text{for every $v\in U_T$}.
% \end{equation*}
% Since the bilinear form is continuous and coercive
% \begin{align*}
% |B_n(u,v)|&\le \frac{1}{\tau_n^2}\|u\|_H\|v\|_H+\|\G\|_{C^0([0,T];B)}\|e u\|_{H}\|e v\|_{H}\le \max\left\{\frac{1}{\tau_n^2},\|\G\|_{C^0([0,T];B)}\right\}\norm{u}_{U_T}\norm{v}_{U_T}&&\quad \text{for $u,v\in U_T$},\\
% B_n(u,u)&\ge \frac{1}{\tau_n^2}\norm{u}_H^2+\beta\norm{e u}_{H}^2\ge \min\left\{\frac{1}{\tau_n^2},\beta\right\}\norm{u}_{U_T}^2&&\quad \text{for $u,v\in U_T$},
% \end{align*}
% we can apply Lax Milgram's lemma to deduce the existence of a unique $u_n^j\in U_T$, such that $u_n^j-z^j_n\in U_n^j$, solution to~\eqref{eq:un}.

We now use equation~\eqref{eq:un} to derive an energy estimate for the family $\{u_n^j\}_{j=1}^n$, which is uniform with respect to $n\in\mathbb N$.

\begin{lemma}\label{lem:reg_est}
Assume~\eqref{eq:data2}--\eqref{eq:C} and~\eqref{eq:G1}--\eqref{eq:data1}. Then there exists a constant $C$, independent of $n\in\mathbb N$, such that 
\begin{equation}\label{d-eq:est}
\max_{j=0,\dots,n}\|\delta u_n^j\|_H+\max_{j=0,\dots,n}\norm{eu_n^j}_H\le C.
\end{equation}
\end{lemma}

\begin{proof}
First, we notice
\begin{align*}
\G_n^0(e u_n^j-eu^0)+\sum_{k=1}^j \tau_n \delta \G_n^{j-k}(e u_n^k-eu^0)=\G_n^{j-1}(e u_n^j-eu^0)+\sum_{k=1}^j\tau_n\delta \G_n^{j-k}(e u_n^k-e u_n^j)\quad\text{for $j=1,\dots,n$}.
\end{align*}
Therefore, equation~\eqref{eq:un} can be written as
\begin{align*}
(\delta^2 u_n^j,v)_H+(\C eu_n^j,ev)_H&+(\G_n^{j-1}(e u_n^j-eu^0),e v)_{H}\hspace{-1pt}+\sum_{k=1}^j\tau_n(\delta \G_n^{j-k}(e u_n^k-e u_n^j),e v)_{H}=(f_n^j,v)_H+(N_n^j,v)_H
\end{align*}
for every $v\in U_n^j$. We fix $i\in\{1,\dots,n\}$. By taking $v:=\tau_n(\delta u_n^j-\delta z_n^j)\in U_n^j $ and summing over $j=1,\dots, i$, we get the following identity
\begin{align}\label{d-en-bal}
\sum_{j=1}^i\tau_n(\delta^2 u_n^j,\delta u_n^j)_H+\sum_{j=1}^i\tau_n(\C eu_n^j,e\delta u_n^j)_H&+\sum_{j=1}^i\tau_n(\G_n^{j-1}(e u_n^j-eu^0),e\delta u_n^j)_{H}\nonumber\\
&+\sum_{j=1}^i\sum_{k=1}^j\tau_n^2(\delta \G_n^{j-k}(e u_n^k-e u_n^j),e\delta u_n^j)_{H}=\sum_{j=1}^i \tau_n L_n^j,
\end{align}
where
\begin{align*}
L_n^j:=(f_n^j,\delta u_n^j-\delta z_n^j)_H&+(N_n^j,\delta u_n^j-\delta z_n^j)_{H_N}+(\delta^2 u_n^j,\delta z_n^j)_H\\
&+(\C eu_n^j,e\delta z_n^j)_H+(\G_n^{j-1}(eu_n^j-eu^0),e\delta z_n^j)_H+\sum_{k=1}^j\tau_n(\delta \G_n^{j-k}(eu_n^k-eu_n^j),e\delta z_n^j)_H.
\end{align*}
By using the identity
\begin{equation*}
|a|^2-a\cdot b=\frac{1}{2}|a|^2-\frac{1}{2}|b|^2+\frac{1}{2}|a-b|^2\quad\text{for every $a,b\in \R^d$}
\end{equation*}
we deduce
\begin{equation*}
\tau_n(\delta^2u_n^j,\delta u_n^j)_H=\norm{\delta u_n^j}_H^2-(\delta u_n^j,\delta u_n^{j-1})_H=\frac{1}{2}\norm{\delta u_n^j}_H^2-\frac{1}{2}\norm{\delta u_n^{j-1}}_H^2+\frac{1}{2}\tau_n^2\norm{\delta^2 u_n^j}_H^2.
\end{equation*}
Therefore
\begin{align}
\sum_{j=1}^i\tau_n(\delta^2u_n^j,\delta u_n^j)_H&=\frac{1}{2}\sum_{j=1}^i\norm{\delta u_n^j}_H^2-\frac{1}{2}\sum_{j=1}^i\norm{\delta u_n^{j-1}}_H^2+\frac{1}{2}\sum_{j=1}^i\tau_n^2\norm{\delta^2 u_n^j}_H^2\nonumber\\
&=\frac{1}{2}\norm{\delta u_n^i}_H^2-\frac{1}{2}\norm{u^1}_H^2+\frac{1}{2}\sum_{j=1}^i\tau_n^2\norm{\delta^2 u_n^j}_H^2.
\end{align}
Similarly, we have
\begin{align}
\sum_{j=1}^i\tau_n(\C e u_n^j,e\delta u_n^j)_{H}
% &=\frac{1}{2}\sum_{j=1}^i(\C e u_n^j,e u_n^j)_{H}-\frac{1}{2}\sum_{j=1}^i(\C e u_n^{j-1},e u_n^{j-1})_{H}+\frac{1}{2}\sum_{j=1}^i\tau_n^2(\C e\delta u_n^j,e\delta u_n^j)_{H}\nonumber\\
&=\frac{1}{2}(\C e u_n^i,e u_n^i)_{H}-\frac{1}{2}(\C e u^0,e u^0)_{H}+\frac{1}{2}\sum_{j=1}^i\tau_n^2(\C e\delta u_n^j,e\delta u_n^j)_{H}.
\end{align}
Moreover, we can write
\begin{align*}
&\tau_n(\G_n^{j-1}(e u_n^j-eu^0),e\delta u_n^j)_{H}=(\G_n^{j-1}(e u_n^j-eu^0),e u_n^j-eu^0)_{H}-(\G_n^{j-1}(e u_n^j-eu^0),e u_n^{j-1}-eu^0)_{H}\\
&=\frac{1}{2}(\G_n^{j-1}(e u_n^j-eu^0),e u_n^j-eu^0)_{H}-\frac{1}{2}(\G_n^{j-1}(e u_n^{j-1}-eu^0),e u_n^{j-1}-eu^0)_{H}+\frac{1}{2}\tau_n^2(\G_n^{j-1}e\delta u_n^j,e\delta u_n^j)_{H}\\
&=\frac{1}{2}(\G_n^j (eu_n^j-eu^0),e u_n^j-eu^0)_{H}-\frac{1}{2}(\G_n^{j-1}(e u_n^{j-1}-eu^0),e u_n^{j-1}-eu^0)_{H}\\
&\hspace{5.2cm}-\frac{1}{2}\tau_n(\delta \G_n^j(e u_n^j-eu^0),e u_n^j-eu^0)_{H}+\frac{1}{2}\tau_n^2(\G_n^{j-1}e\delta u_n^j,e\delta u_n^j)_{H}.
\end{align*}
As consequence of this we obtain
\begin{align}
&\sum_{j=1}^i\tau_n(\G_n^{j-1}(e u_n^j-eu^0),e\delta u_n^j)_{H}\nonumber\\
&=\frac{1}{2}\sum_{j=1}^i(\G_n^j(e u_n^j-eu^0),e u_n^j-eu^0)_{H}-\frac{1}{2}\sum_{j=1}^i(\G_n^{j-1}(e u_n^{j-1}-eu^0),e u_n^{j-1}-eu^0)_{H}\nonumber\\
&\quad-\frac{1}{2}\sum_{j=1}^i\tau_n(\delta \G_n^j (e u_n^j-eu^0),e u_n^j-eu^0)_{H}+\frac{1}{2}\sum_{j=1}^i\tau_n^2(\G_n^{j-1}e\delta u_n^j,e\delta u_n^j)_{H}\nonumber\\
&=\frac{1}{2}(\G_n^i (e u_n^i-eu^0),e u_n^i-eu^0)_{H}-\frac{1}{2}\sum_{j=1}^i\tau_n(\delta \G_n^j(e u_n^j-eu^0),e u_n^j-eu^0)_{H}+\frac{1}{2}\sum_{j=1}^i\tau_n^2(\G_n^{j-1}e\delta u_n^j,e\delta u_n^j)_{H}.
\end{align}
Finally, let us consider the term
\begin{equation*}
\sum_{j=1}^i\sum_{k=1}^j\tau_n^2(\delta \G_n^{j-k}(e u_n^k-e u_n^j),e\delta u_n^j)_{H}=\sum_{k=1}^i\sum_{j=k}^i\tau_n^2(\delta \G_n^{j-k}(e u_n^k-e u_n^j),e\delta u_n^j)_{H}.
\end{equation*}
We can write
\begin{align*}
&\sum_{j=k}^i\tau_n^2(\delta \G_n^{j-k}(e u_n^k-e u_n^j),e\delta u_n^j)_{H}=-\sum_{j=k}^i\tau_n(\delta \G_n^{j-k}(e u_n^j-e u_n^k),e u_n^j-e u_n^{j-1})_{H}\\
&=-\sum_{j=k}^i\tau_n(\delta \G_n^{j-k}(e u_n^j-e u_n^k),e u_n^j-e u_n^k)_{H}+\sum_{j=k}^i\tau_n(\delta \G_n^{j-k}(e u_n^j-e u_n^k),e u_n^{j-1}-e u_n^k)_{H}\\
&=-\frac{1}{2}\sum_{j=k}^i\tau_n(\delta \G_n^{j-k}(e u_n^j-e u_n^k),e u_n^j-e u_n^k)_{H}+\frac{1}{2}\sum_{j=k}^i\tau_n(\delta \G_n^{j-k}(e u_n^{j-1}-e u_n^k),e u_n^{j-1}-e u_n^k)_{H}\\
&\hspace{7cm}-\frac{1}{2}\sum_{j=k}^i\tau_n^3(\delta \G_n^{j-k}e\delta u_n^j,e\delta u_n^j)_{H}\\
&=-\frac{1}{2}\sum_{j=k}^i\tau_n(\delta \G_n^{j-k+1}(e u_n^j-e u_n^k),e u_n^j-e u_n^k)_{H}+\frac{1}{2}\sum_{j=k}^i\tau_n(\delta \G_n^{j-k}(e u_n^{j-1}-e u_n^k),e u_n^{j-1}-e u_n^k)_{H}\\
&\quad+\frac{1}{2}\sum_{j=k}^i\tau_n^2(\delta^2 \G_n^{j-k+1}(e u_n^j-e u_n^k),e u_n^j-e u_n^k)_{H}-\frac{1}{2}\sum_{j=k}^i\tau_n^3(\delta \G_n^{j-k}e\delta u_n^j,e\delta u_n^j)_{H}\\
&=\frac{1}{2}\sum_{j=k}^i\tau_n^2(\delta^2 \G_n^{j-k+1}(e u_n^j-e u_n^k),e u_n^j-e u_n^k)_{H}-\frac{1}{2}\sum_{j=k}^i\tau_n^3(\delta \G_n^{j-k}e\delta u_n^j,e\delta u_n^j)_{H}\\
&\hspace{7.25cm}-\frac{1}{2}\tau_n(\delta \G_n^{i-k+1}(e u_n^i-e u_n^k),e u_n^i-e u_n^k)_{H}
\end{align*}
because $\delta \G_n^0=0$. Therefore, we deduce
\begin{align}
&\sum_{j=1}^i\sum_{k=1}^j\tau_n^2(\delta \G_n^{j-k}(e u_n^k-e u_n^j),e\delta u_n^j)_{H}\nonumber\\
&=\frac{1}{2}\sum_{k=1}^i\sum_{j=k}^i\tau_n^2(\delta^2 \G_n^{j-k+1}(e u_n^j-e u_n^k),e u_n^j-e u_n^k)_{H}-\frac{1}{2}\sum_{k=1}^i\sum_{j=k}^i\tau_n^3(\delta \G_n^{j-k}e\delta u_n^j,e\delta u_n^j)_{H}\nonumber\\
&\quad-\frac{1}{2}\sum_{k=1}^i\tau_n(\delta \G_n^{i-k+1}(e u_n^i-e u_n^k),e u_n^i-e u_n^k)_{H}\nonumber\\
&=\frac{1}{2}\sum_{j=1}^i\sum_{k=1}^j\tau_n^2(\delta^2 \G_n^{j-k+1}(e u_n^j-e u_n^k),e u_n^j-e u_n^k)_{H}-\frac{1}{2}\sum_{j=1}^i\sum_{k=1}^j\tau_n^3(\delta \G_n^{j-k}e\delta u_n^j,e\delta u_n^j)_{H}\nonumber\\
&\quad-\frac{1}{2}\sum_{j=1}^i\tau_n(\delta \G_n^{i-j+1}(e u_n^i-e u_n^j),e u_n^i-e u_n^j)_{H}\label{scomp-3}.
\end{align}

By combining together~\eqref{d-en-bal}--\eqref{scomp-3}, we obtain for $i=1,\dots,n$ the following discrete energy equality 
\begin{align}
&\frac{1}{2}\|\delta u_n^i\|_H^2+\frac{1}{2}(\C eu_n^i,eu_n^i)_H+\frac{1}{2}(\G_n^i (e u_n^i-eu^0),e u_n^i-eu^0)_{H}-\frac{1}{2}\sum_{j=1}^i\tau_n(\delta \G_n^{i-j+1}(e u_n^i-e u_n^j),e u_n^i-e u_n^j)_{H}\nonumber\\
&\quad-\frac{1}{2}\sum_{j=1}^i\tau_n(\delta \G_n^j(e u_n^j-eu^0),e u_n^j-eu^0)_{H}+\frac{1}{2}\sum_{j=1}^i\sum_{k=1}^j\tau_n^2(\delta^2 \G_n^{j-k+1}(e u_n^j-e u_n^k),e u_n^j-e u_n^k)_{H}\nonumber\\
&\quad+\frac{\tau^2_n}{2}\left(\sum_{j=1}^i\|\delta^2u_n^j\|_H^2+\sum_{j=1}^i(\C e\delta u_n^j,e\delta u_n^j)_{H}+\sum_{j=1}^i(\G_n^{j-1}e\delta u_n^j,e\delta u_n^j)_{H}-\sum_{j=1}^i\sum_{k=1}^j\tau_n(\delta \G^{j-k}e\delta u_n^j,e\delta u_n^j)_{H}\right)\nonumber\\
&\hspace{8cm}=\frac{1}{2}\|u^1\|_H^2+\frac{1}{2}(\C e u^0,e u^0)_{H}+\sum_{j=1}^i\tau_n L_n^j.\label{d-en-bal-scomp}
\end{align}
By our assumptions on $\G$ we deduce
\begin{align*}
&\G_n^j(x)\xi\cdot\xi\ge 0&&\quad\text{for a.e. $x\in\Omega$ and every $\xi\in\R^d$ and $j=0,\dots,n$},\\
&\delta \G_n^j(x)\xi\cdot\xi=\dashint_{(j-1)\tau_n}^{j\tau_n}\dot \G(r,x)\xi\cdot\xi\,\de r\le 0&&\quad\text{for a.e. $x\in\Omega$ and every $\xi\in\R^d$ and $j=1,\dots,n$},\\
&\delta^2\G_n^j(x)\xi\cdot\xi
% &=\frac{1}{\tau_n^2}\int_{(j-1)\tau_n}^{j\tau_n}\dot \G(r,x)\xi\cdot\xi\,\de r-\frac{1}{\tau_n^2}\int_{(j-2)\tau_n}^{(j-1)\tau_n}\dot \G(r,x)\xi\cdot\xi\,\de r\\
% &=\frac{1}{\tau_n^2}\int_{(j-1)\tau_n}^{j\tau_n}(\dot \G(r,x)-\dot \G(r-\tau_n,x))\xi\cdot\xi\,\de r\\
=\dashint_{(j-1)\tau_n}^{j\tau_n}\dashint_{r-\tau_n}^r\ddot \G(s,x)\xi\cdot\xi\,\de s\,\de r\ge 0&&\quad\text{for a.e. $x\in\Omega$ and every $\xi\in\R^d$ and $j=2,\dots,n$}.
\end{align*}
% Therefore for every $i=1,\dots,n$ we get
% \begin{align*}
% &\frac{1}{2}(\G_n^i (e u_n^i-eu^0),e u_n^i-eu^0)_{H}\ge 0,&&\quad\frac{1}{2}\sum_{j=1}^i\tau_n(\delta \G_n^{i-j+1}(eu_n^i-eu_n^j),eu_n^i-eu_n^j)_{H}\le 0,\\
% & \frac{1}{2}\sum_{j=1}^i\tau_n(\delta \G_n^j(eu_n^j-eu^0),eu_n^j-eu^0)_{H}\le 0,&&\quad\frac{1}{2}\sum_{j=1}^i\sum_{k=1}^j\tau_n^2(\delta^2 \G_n^{j-k+1}(eu_n^j-eu_n^k),eu_n^j-eu_n^k)_{H}\ge 0,\\
% & \frac{1}{2}\sum_{j=1}^i\tau_n\|\delta^2u_n^i\|_H^2\ge 0&&\quad\frac{1}{2}\sum_{j=1}^i\tau_n(\G_n^{j-1}e\delta u_n^j,e\delta u_n^j)_{H}\ge 0,\\
% &\frac{1}{2}\sum_{j=1}^i\sum_{k=1}^j\tau_n^2(\delta \G^{j-k}e\delta u_n^j,e\delta u_n^j)_{H}\le 0. & &
% \end{align*}
Hence, thanks to~\eqref{d-en-bal-scomp}, for every $i=1,\dots,n$ we can write 
\begin{equation}\label{d-en-in}
\frac{1}{2}\|\delta u_n^i\|_H^2+\frac{1}{2}(\C eu_n^i,eu_n^i)_H\le\frac{1}{2}\|u^1\|_H^2+\frac{1}{2}(\C e u^0,e u^0)_{H}+\sum_{j=1}^i\tau_n L_n^j.
\end{equation}

Let us estimate the right-hand side in~\eqref{d-en-in} from above. We set
\begin{equation*}
 K_n:=\max_{j=0,..,n}\|\delta u_n^j\|_H,\quad E_n:=\max_{j=0,..,n}\|e u_n^j\|_H.
\end{equation*}
Therefore, we have the following bounds
\begin{align}
\left|\sum_{j=1}^i \tau_n(f_n^j,\delta u_n^j)_H\right| &
% \leq K_n\sum_{j=1}^i\tau_n\|f_n^j\|_H
\leq \sqrt{T}\|f\|_{L^2(0,T;H)}K_n,\\
\left|\sum_{j=1}^i \tau_n(f_n^j,\delta z_n^j)_H\right| 
&\leq \|f\|_{L^2(0,T;H)}\|\dot z\|_{L^2(0,T;H)},\\
 \left|\sum_{j=1}^i \tau_n(\C eu_n^j,e\delta z_n^j)_H\right|&\leq \|\C\|_B\|e\dot z\|_{L^1(0,T;H)}E_n,\\
\left|\sum_{j=1}^i\tau_n (\G_n^{j-1}(eu_n^j-eu^0),e\delta z_n^j)_H\right|&
% \leq 2\|\G\|_{C^0([0,T];B)}E_n\sum_{j=1}^i \tau_n\|e\delta z_n^j\|_{H}
\leq 2\|\G\|_{C^0([0,T];B)}\|e\dot z\|_{L^1(0,T;H)}E_n.
\end{align}
Notice that the following discrete integrations by parts hold
\begin{align}
&\sum_{j=1}^i \tau_n(\delta^2 u_n^j,\delta z^j_n)_H=(\delta u_n^i,\delta z_n^i)_H-(\delta u_n^0,\delta z_n^0)_H-\sum_{j=1}^i\tau_n (\delta u_n^{j-1},\delta^2 z_n^j)_H,\label{dis-part-1}\\
&\sum_{j=1}^i \tau_n(N_n^j,\delta u^j_n)_{H_N}=(N_n^i,u_n^i)_{H_N}-(N_n^0,u^0_n)_{H_N}-\sum_{j=1}^i\tau_n (\delta N_n^{j},u_n^{j-1})_{H_N},\label{dis-part-2}\\
&\sum_{j=1}^i \tau_n(N_n^j,\delta z^j_n)_{H_N}=(N_n^i,z^i_n)_{H_N}-(N_n^0,z^0_n)_{H_N}-\sum_{j=1}^i\tau_n (\delta N_n^{j},z_n^{j-1})_{H_N}.\label{dis-part-3}
\end{align}
By means of~\eqref{dis-part-1} we can write
\begin{align}
\left| \sum_{j=1}^i (\delta^2 u_n^j,\delta z^j_n)_H\right|&\leq \|\delta u_n^i\|_H\|\delta z_n^i\|_H+\|\delta u^0_n\|_H\|\delta z^0_n\|_H+ \sum_{j=1}^i \tau_n \|\delta u_n^{j-1}\|_H\|\delta^2 z_n^j\|_H\nonumber\\
&\leq (2\|\dot z\|_{C^0([0,T];H)}+\|\ddot z\|_{L^1(0,T;H)})K_n.
% \leq 2\|\dot z\|_{C^0([0,T];H)}K_n+K_n\sum_{j=1}^i \tau_n \|\delta^2 z_n^{j}\|_H
\end{align}
Moreover, thanks to
\begin{equation}
\|u_n^i\|_{U_T}\le \norm{u_n^i}_H+E_n\leq \sum_{j=1}^i\tau_n \norm{\delta u_n^j}_H+\norm{u^0}_H+E_n\leq T K_n+E_n+\|u^0\|_H\quad\text{for $i=0,\dots,n$}
\end{equation}
and to~\eqref{dis-part-2} we obtain
\begin{align}
\left| \sum_{j=1}^i \tau_n(N_n^j,\delta u^j_n)_{H_N}\right|&\leq \|N_n^i\|_{H_N}\|u_n^i\|_{H_N}+\|N^0_n\|_{H_N}\|u^0_n\|_{H_N}+\sum_{j=1}^i\tau_n \|\delta N_n^{j}\|_{H_N}\|u_n^{j-1}\|_{H_N}\nonumber\\
&\leq C_{tr}\|N\|_{C^0([0,T];H_N)}(\|u_n^i\|_{U_T}+\|u^0_n\|_{U_T})+C_{tr}\sum_{j=1}^i\tau_n \|\delta N_n^{j}\|_{H_N}\|u_n^{j-1}\|_{U_T}\nonumber\\
&\leq C_{tr}\left(2\|N\|_{C^0([0,T];H_N)}+\|\dot N\|_{L^1(0,T;H_N)}\right)(E_n+TK_n+\|u^0\|_H).
\end{align}
Similarly, by~\eqref{dis-part-3} we obtain
\begin{equation}
\left| \sum_{j=1}^i \tau_n(N_n^j,\delta z^j_n)_{H_N}\right|\leq %\|N_n^i\|_{H_N}\|z_n^i\|_{H_N}+\|N^0_n\|_{H_N}\|z^0_n\|_{H_N}+\sum_{j=1}^i\tau_n \|\delta N_n^{j}\|_{H_N}\|z_n^{j-1}\|_{H_N}\nonumber\\
% % &\leq 2C_{tr}\|N\|_{C^0([0,T];H_N)}\|z\|_{C^0([0,T];U_0)}+C_{tr}\sum_{j=1}^i\tau_n \|\delta N_n^{j}\|_{H_N}\|z_n^{j-1}\|_{U_0}\nonumber\\
C_{tr}\left(2\|N\|_{C^0([0,T];H_N)}+\|\dot N\|_{L^1(0,T;H_N)}\right)\|z\|_{C^0([0,T];U_0)}.
\end{equation}
Finally, we have
\begin{align}\label{rhs-6}
\left|\sum_{j=1}^i\sum_{k=1}^j \tau ^2_n(\delta \G_n^{j-k}(eu_n^k-eu_n^j),e\delta z_n^j)_H\right|&\leq \sum_{j=1}^i\sum_{k=1}^j\tau_n^2\|\delta \G_n^{j-k}\|_B\|eu_n^k-eu_n^j\|_H\|e\delta z_n^j\|_H\nonumber\\
% &\leq 2TE_n\|\dot \G\|_{C^0([0,T];B)}\sum_{j=1}^i\tau_n\|e\delta z_n^j\|_H\nonumber\\
&\le 2T\|\dot \G\|_{C^0([0,T];B)}\|e\dot z\|_{L^1(0,T;H)}E_n.
\end{align}
By considering~\eqref{d-en-in}--\eqref{rhs-6} and using~\eqref{eq:C}, we obtain the existence of a constant $C_1=C_1(z,N,f,u^0,\C,\G)$ such that
\begin{align*}
&\|\delta u_n^i\|^2_H+\gamma\|eu_n^i\|_H^2\leq \|u^1\|_H^2+\|\C\|_{B}\|e u^0\|_{H}^2+ C_1\left(1+K_n+E_n\right)\quad\text{for $i=1,\dots,n$}.
\end{align*} 
In particular, since the right-hand side is independent of $i$, $u_n^0=u^0$ and $\delta u_n^0=u^1$, there exists another constant $C_2=C_2(z,N,f,u^0,u^1,\C,\G)$ for which we have
\begin{equation*}
K_n^2+ E_n^2\le C_2(1+ K_n+ E_n) \quad\text{for every $n\in\mathbb N$}.
\end{equation*}
This implies the existence of a constant $C=C(z,N,f,u^0,u^1,\C,\G)$ independent of $n\in\N$ such that 
\begin{equation*}
\|\delta u_n^j\|_H+\|eu_n^j\|_H\leq K_n+E_n\leq C\qquad \text{for every $j=1,\dots,n$ and $n\in \N$},
\end{equation*}
which gives~\eqref{d-eq:est}.
\end{proof}

A first consequence of Lemma~\ref{lem:reg_est} is the following uniform estimate on the family $\{\delta^2 u_n^j\}_{j=1}^n$. 

\begin{corollary}
Assume~\eqref{eq:data2}--\eqref{eq:C} and~\eqref{eq:G1}--\eqref{eq:data1}. Then there exists a constant $\tilde C$, independent of $n\in\mathbb N$, such that 
\begin{equation}\label{sec-der-bound}
\sum_{j=1}^n\tau_n\|\delta^2 u^j_n\|^2_{(U^D_0)'}\leq\tilde C.
\end{equation}
\end{corollary}

\begin{proof}
Thanks to equation~\eqref{eq:un} and to Lemma~\ref{lem:reg_est}, for every $j=1,\dots,n$ and $v\in U^D_0 \subseteq U_n^j$ with $\|v\|_{U_0}\leq 1$ we have
\begin{align*}
|(\delta^2 u^j_n,v)_H|\leq C\left(\|\C\|_B + 2\|\G \|_{C^0([0,T];B)} +2T\|\dot \G\|_{C^0([0,T];B)}\right) +\| f^j_n\|_H+C_{tr}\|N\|_{C^0([0,T];H_N)}.
\end{align*}
By taking the supremum over $v\in U^D_0$ with $\norm{v}_{U_0}\leq1$ we obtain
\begin{equation*}
\|\delta^2 u^j_n\|_{(U^{D}_0)'}^2\leq 3C^2\left(\|\C\|_B + 2\|\G \|_{C^0([0,T];B)} +2T\|\dot \G\|_{C^0([0,T];B)}\right)^2+3\| f^j_n\|_H^2+3C_{tr}^2\|N\|_{C^0([0,T];H_N)}^2.
\end{equation*}
We multiply this inequality by $\tau_n$ and we sum over $j=1,\dots,n$ to get
% \begin{align*}
% \sum_{j=1}^n\tau_n\|\delta^2 u^j_n\|^2_{(U^D_0)'}&\leq3C^2T(\|\C\|_B + 2\|\G \|_{C^0([0,T];B)} +2T\|\dot \G\|_{C^0([0,T];B)})^2\\
% &\quad+ 3\| f\|_{L^2(0,T;H)}^2+3T\|N\|_{C^0([0,T];H_N)}^2.
% \end{align*}
% Hence, we deduce 
\eqref{sec-der-bound}.
\end{proof}

We now want to pass to the limit into equation~\eqref{eq:un} to obtain a generalized solution to system~\eqref{eq:Gsystem}. 
Let us recall the following result, whose proof can be found for example in~\cite{DL}.

\begin{lemma}\label{lem:wc}
Let $X,Y$ be two reflexive Banach spaces such that $X\hookrightarrow Y$ continuously. Then 
$$L^{\infty}(0, T;X)\cap C^0_w([0, T];Y)= C^0_w([0, T];X).$$
\end{lemma}

Let us define the following sequences of functions which are an approximation of the generalized solution:
\begin{align*}
&u_n(t)=u_n^i+(t-i\tau_n)\delta u_n^i& &\text{for $t\in[(i-1)\tau_n,i\tau_n]$ and $i=1,\dots,n$},\\
&u_n^+(t)=u_n^i& &\text{for $t\in((i-1)\tau_n,i\tau_n]$ and $i=1,\dots,n$},& & u_n^+(0)=u_n^0,\\
&u_n^-(t)=u_n^{i-1}& &\text{for $t\in[(i-1)\tau_n,i\tau_n)$ and $i=1,\dots,n$},& & u_n^-(T)=u_n^n.
\end{align*}
Moreover, we consider also the sequences
\begin{align*}
&\tilde u_n(t)=\delta u_n^i+(t-i\tau_n)\delta^2 u_n^i& &\text{for $t\in[(i-1)\tau_n,i\tau_n]$ and $i=1,\dots,n$},\\
&\tilde u_n^+(t)=\delta u_n^i& &\text{for $t\in((i-1)\tau_n,i\tau_n]$ and $i=1,\dots,n$},& & \tilde u_n^+(0)=\delta u_n^0,\\
&\tilde u_n^-(t)=\delta u_n^{i-1}& &\text{for $t\in[(i-1)\tau_n,i\tau_n)$ and $i=1,\dots,n$},& & \tilde u_n^-(T)=\delta u_n^n,
\end{align*}
which approximate the first time derivative of the generalized solution. In a similar way, we define also $f_n^+$, $N_n^+$, $\tilde N_n^+$, $z_n^\pm$, $\tilde z_n$, $\tilde z_n^+$, $\G_n^\pm$, $\tilde\G_n$, $\tilde\G_n^+$. Thanks to the uniform estimates of Lemma~\ref{lem:reg_est} we derive the following compactness result:

\begin{lemma}\label{lem:conv}
Assume~\eqref{eq:data2}--\eqref{eq:C} and~\eqref{eq:G1}--\eqref{eq:data1}. There exists a function $u\in\mathcal C_w\cap H^2(0,T;(U_0^D)')$ such that, up to a not relabeled subsequence
\begin{equation}\label{weak-conv}
u_n\xrightharpoonup[n\to \infty]{H^1(0, T;H)}u,\quad  u^\pm_n \xstararrow{n\to \infty}{L^\infty(0, T;U_T)}u,\quad   \tilde u_n \xrightharpoonup[n\to \infty]{H^1(0, T;(U_0^D)')}\dot u,\quad  \tilde{u}^\pm_n \xstararrow{n\to \infty}{L^\infty(0, T;H)}\dot{u},
\end{equation}
and for every $t\in[0,T]$
\begin{equation}\label{point-conv}
   u_n^\pm(t)\xrightharpoonup[n\to\infty]{U_T}u(t),\quad \tilde u_n^\pm(t)\xrightharpoonup[n\to\infty]{H}\dot u(t).
\end{equation}
\end{lemma}

\begin{proof}
Thanks to Lemma~\ref{lem:reg_est} and the estimate~\eqref{sec-der-bound}, the sequences 
\begin{align*}
&\{u_n\}_n\subseteq L^\infty(0, T;U_T)\cap H^1(0, T;H),
& &\{\tilde u_n\}_n\subseteq L^\infty(0,T;H)\cap H^1(0,T;(U_0^D)'),\\
&\{u_n^\pm\}_n\subseteq L^\infty(0, T;U_T),
& &\{\tilde u_n^\pm\}_n\subseteq L^\infty(0,T;H),
\end{align*}
are uniformly bounded with respect to $n\in\mathbb N$.
% Indeed, by~\eqref{d-eq:est} and~\eqref{tool-rhs-2} there exists a constant $\bar C$ independent of $n\in\mathbb N$ such that $\|u_n^i\|_{U_T}\leq \bar C$ for every $i=1,\dots,n$. Therefore
% \begin{align*}
% &\|u_n\|_{L^{\infty}(0,T;U_T)}\leq \max_{k=1,..,n}\ \sup_{t\in [(k-1)\tau_n,k\tau_n]}\|\left(1-k+t\tau^{-1}_n\right)u_n^k+\left(k-t r ^{-1}_n\right)u_n^{k-1}\|_{U_T}\leq \bar C,\\
% &\|\dot u_n\|_{L^{\infty}(0,T;H)}\leq \max_{k=1,..,n}\|\delta u_n^k\|_{H}\leq C,\\
% &\|u_n^+\|_{L^{\infty}(0,T;U_T)}\leq \max_{k=1,..,n}\|\delta u_n^k\|_H\leq C,\\
% \end{align*} 
By Banach-Alaoglu's theorem and Lemma~\ref{lem:wc} there exist two functions $u\in C^0_w([0,T];U_T)\cap H^1(0,T;H)$ and $v\in C^0_w([0,T];H)\cap H^1(0,T;(U_0^D)')$, such that, up to a not relabeled subsequence
\begin{alignat}{4}
&u_n \xrightharpoonup[n\to \infty]{H^1(0,T;H)}u,&& \quad u_n \xstararrow{n\to \infty}{L^\infty(0, T;U_T)}u,&&\quad \tilde u_n \xrightharpoonup[n\to \infty]{H^1(0, T;(U_0^D)')} v,&& \quad \tilde u_n \xstararrow{n\to \infty}{L^\infty(0, T;H)} v.\label{weak-conv-2}
\end{alignat}
Thanks to~\eqref{sec-der-bound} we get
\begin{align*}
\|\dot u_n-\tilde u_n\|_{L^2(0,T;(U_0^D)')}^2\leq \tilde C \tau_n^2\xrightarrow[n\to\infty]{}0,
\end{align*}
therefore we deduce that $v=\dot u$. Moreover, by using~\eqref{d-eq:est} and~\eqref{sec-der-bound} we have
\begin{align*}
&\|u_n^\pm-u_n\|_{L^\infty(0,T;H)}\le C \tau_n\xrightarrow[n\to\infty]{}0, \qquad\|\tilde u_n^\pm-\tilde u_n\|_{L^2(0,T;(U_0^D)')}^2\le \tilde C \tau_n^2\xrightarrow[n\to\infty]{}0.
\end{align*}
We combine the previous convergences with~\eqref{weak-conv-2} to derive
\begin{equation*}
u^\pm_n \xstararrow{n\to \infty}{L^\infty(0, T;U_T)} u,\qquad \tilde{u}^\pm_n \xstararrow{n\to \infty}{L^\infty(0, T;H)}\dot{u}.
\end{equation*}

By~\eqref{weak-conv-2} for every $t\in[0,T]$ we have
\begin{equation*}
u_n(t)\xrightharpoonup[n\to\infty]{U_T}u(t),\qquad \tilde u_n(t)\xrightharpoonup[n\to\infty]{H}\dot u(t).
\end{equation*}
Again, thanks to~\eqref{d-eq:est} and~\eqref{sec-der-bound}, for every $t\in[0,T]$ we get
\begin{align*}
&\|u_n^\pm(t)\|_{U_T}\le C,& &\|u_n^\pm(t)-u_n(t)\|_H\le C\tau_n\xrightarrow[n\to\infty]{}0,\\
&\|\tilde u_n^\pm(t)\|_{H}\le C,& &\|\tilde u_n^\pm(t)-\tilde u_n(t)\|_{(U_0^D)'}^2\le \tilde C\tau_n\xrightarrow[n\to\infty]{}0,
\end{align*}
which imply~\eqref{point-conv}. Finally, observe that for every $t\in[0,T]$
\begin{equation*}
 u_n^-(t)\in U_t,\qquad u_n^-(t)\xrightharpoonup[n\to\infty]{U_T}u(t).
\end{equation*}
Therefore, $u(t)\in U_t$ for every $t\in[0,T]$ since $U_t$ is a closed subspace of $U_T$. Hence, $u\in \mathcal C_w$.
\end{proof}

Let us check that the limit function $u$ defined before satisfies the boundary and initial conditions.

\begin{corollary}\label{in-con}
Assume~\eqref{eq:data2}--\eqref{eq:C} and~\eqref{eq:G1}--\eqref{eq:data1}. Then the function $u\in\mathcal C_w$ of Lemma~\ref{lem:conv} satisfies for every $t\in[0,T]$ the condition $u(t)-z(t)\in U_t^D$, and it assumes the initial conditions $u(0)=u^0$ in $U_0$ and $\dot u(0)=u^1$ in $H$. 
\end{corollary}

\begin{proof}
By~\eqref{weak-conv} we have
\begin{equation*}
u^0=u_n(0)\xrightharpoonup[n\to\infty]{U_T}u(0),\qquad u^1=\tilde u_n(0)\xrightharpoonup[n\to\infty]{H}\dot u(0).
\end{equation*}
Hence, $u\in \mathcal C_w$ satisfies $u(0)=u^0$ in $U_0$ and $\dot u(0)=u^1$ in $H$. Moreover, since $z\in C^0([0,T];U_0)$ and thanks to~\eqref{point-conv}, we have for every $t\in[0,T]$ 
\begin{equation*}
u_n^-(t)-z_n^-(t)\in U_t^D,\qquad u_n^-(t)-z_n^-(t)\xrightharpoonup[n\to\infty]{U_T}u(t)-z(t).
\end{equation*}
Thus, $u(t)-z(t)\in U_t^D$ for every $t\in[0,T]$ because $U_t^D$ is a closed subspace of $U_T$.
\end{proof} 

\begin{lemma}\label{lem:sol}
Assume~\eqref{eq:data2}--\eqref{eq:C} and~\eqref{eq:G1}--\eqref{eq:data1}. Then the function $u\in\mathcal C_w$ of Lemma~\ref{lem:conv} is a generalized solution to system~\eqref{eq:Gsystem}.
\end{lemma}

\begin{proof}
We only need to prove that the function $u\in\mathcal C_w$ satisfies~\eqref{eq:Ggen}. We fix $n\in\mathbb N$ and a function $\varphi\in \Cc$. Let us consider
\begin{alignat*}{4}
&\varphi_n^j:=\varphi(j\tau_n)&&\qquad\text{for $j=0,\dots,n$},&&\qquad\delta \varphi_n^j:=\frac{\varphi_n^j-\varphi_n^{j-1}}{\tau_n}&&\qquad \text{for $j=1,\dots,n$},
\end{alignat*}
and, as we did before for the family $\{u_n^j\}_{j=1}^n$, we define the approximating sequences $\{\varphi_n^+\}_n$ and $\{\tilde\varphi_n^+\}_n$. If we use $\tau_n\varphi_n^j\in U_n^j$ as a test function in~\eqref{eq:un}, after summing over $j=1,...,n$, we get
\begin{align}\label{pass-limit}
\sum_{j=1}^n\tau_n(\delta^2u_n^j,\varphi^j_n)_H&+\sum_{j=1}^n\tau_n(\C eu_n^j,e\varphi^j_n)_{H}+\sum_{j=1}^n\tau_n(\G_n^0 (eu_n^j-eu^0),e\varphi^j_n)_{H}\nonumber\\
&+\sum_{j=1}^n\sum_{k=1}^j\tau^2_n(\delta \G_n^{j-k}(e u_n^k-eu^0),e \varphi_n^j)_{H}=\sum_{j=1}^n\tau_n(f_n^j,\varphi^j_n)_H+\sum_{j=1}^n\tau_n(N_n^j,\varphi^j_n)_{H_N}.
\end{align}
By means of a time discrete integration by parts we obtain
\begin{equation*}
\sum_{j=1}^n \tau_n(\delta^2 u^j_n,\varphi^j_n)_H
= -\sum_{j=1}^n\tau_n(\delta u^{j-1}_n,\delta \varphi^j_n)_H=-\int_0^T(\tilde{u}^-_n(t),\tilde\varphi_n^+(t))_H \,\de t,
\end{equation*}
and since $\delta \G_n^0=0$ and $\varphi_n^0=\varphi_n^n=0$ we get
\begin{align*}
&\sum_{j=1}^n\tau_n(\G_n^0 (eu_n^j-eu^0),e\varphi^j_n)_{H}+\sum_{j=1}^n\sum_{k=1}^j\tau^2_n(\delta \G_n^{j-k}(e u_n^k-eu^0),e \varphi_n^j)_{H}\\
&=-\sum_{j=1}^{n-1}\sum_{k=1}^j\tau_n^2 (\G_n^{j-k}(eu_n^k-eu^0),e\delta \varphi_n^{j+1})_H=-\int_0^{T-\tau_n}\hspace{-2pt}\int_0^{t_n}(\G_n^-(t_n-r)(eu_n^+(r)-eu^0),e\tilde\varphi_n^+(t+\tau_n))_H\,\de r\,\de t,
\end{align*}
where $t_n:=\left\lceil\frac{t}{\tau_n}\right\rceil\tau_n$ for $t\in(0,T)$ and $\lceil x\rceil$ is the superior integer part of the number $x$. Thanks to~\eqref{pass-limit} we deduce
\begin{align}\label{pass-limit-2}
&-\int_0^T(\tilde{u}^-_n(t),\tilde\varphi_n^+(t))_H \,\de t -\int_0^{T-\tau_n}\int_0^{t_n}(\G_n^-(t_n-r)(eu_n^+(r)-eu^0),e\tilde\varphi_n^+(t+\tau_n))_H\,\de r\,\de t\nonumber\\
&\hspace{2.5cm}+\int_0^T(\C eu_n^+(t),e\varphi_n^+(t))_H\,\de t=\int_0^T(f^+_n(t),\varphi^+_n(t))_H \,\de t+\int_0^T(N^+_n(t),\varphi^+_n(t))_{H_N} \,\de t.
\end{align}
We use~\eqref{weak-conv} and the following convergences 
\begin{align*}
&\varphi^+_n\xrightarrow[n\to\infty]{L^2(0,T;U_T)}\varphi, \quad \tilde\varphi_n^+\xrightarrow[n\to\infty]{L^2(0,T;H)}\dot{\varphi},\quad f^+_n\xrightarrow[n\to\infty]{L^2(0,T;H)}f, \quad N_n^+\xrightarrow[n\to\infty]{L^2(0,T;H_N)}N,
\end{align*}
to derive 
\begin{align*}
\int_0^T(\tilde{u}^-_n(t),\tilde\varphi_n^+(t))_H \,\de t&\xrightarrow[n\to\infty]{}\int_0^T(\dot u(t),\dot\varphi(t))_H \,\de t,\\
\int_0^T(\C eu_n^+(t),e\varphi_n^+(t))_H\,\de t &\xrightarrow[n\to\infty]{}\int_0^T(\C eu(t),e\varphi(t))_H\,\de t,\\
\int_0^T(f^+_n(t),\varphi^+_n(t))_H\,\de t&\xrightarrow[n\to\infty]{} \int_0^T(f(t),\varphi(t))_H\,\de t,\\
\int_0^T(N^+_n(t),\varphi^+_n(t))_{H_N} \,\de t&\xrightarrow[n\to\infty]{} \int_0^T(N(t),\varphi(t))_{H_N} \,\de t.
\end{align*}
Moreover, for every fixed $t\in(0,T)$
\begin{equation}\label{arg-1}
\chi_{[0,T-\tau_n]}(t)\chi_{[0,t_n]}(\,\cdot\,)\G_n^-(t_n-\,\cdot\,)e\tilde\varphi_n^+(t+\tau_n)\xrightarrow[n\to\infty]{L^2(0,T;H)}\chi_{[0,T]}(t)\chi_{[0,t]}(\,\cdot\,)\G(t-\,\cdot\,)e\dot\varphi(t),
\end{equation}
which together with~\eqref{weak-conv} gives
\begin{align}
\chi_{[0,T-\tau_n]}(t)\int_0^{t_n}(\G_n^-(t_n-r)(eu_n^+(r)&-eu^0),e\tilde\varphi_n^+(t+\tau_n))_H\,\de r\nonumber\\
&\xrightarrow[n\to\infty]{}\chi_{[0,T]}(t)\int_0^t(\G(t-r)(eu(r)-eu^0),e\dot\varphi(t))_H\,\de r.
\end{align}
By~\eqref{d-eq:est} for every $t\in(0,T)$ we deduce
\begin{equation}\label{arg-3}
\hspace{-3pt}\left|\chi_{[0,T-\tau_n]}(t)\int_0^{t_n}(\G_n^-(t_n-r)(eu_n^+(r)-eu^0),e\tilde\varphi_n^+(t+\tau_n))_H\,\de r\right|\le 2T\|\G\|_{C^0([0,T];B)} C\|e\dot\varphi\|_{C^0([0,T];H)}.
\end{equation}
Therefore, we can use the dominated convergence theorem to pass to the limit in the double integral of~\eqref{pass-limit-2}, and we obtain that $u$ satisfies~\eqref{eq:Ggen} for every function $\varphi\in \Cc$.
\end{proof}

Now we want to deduce an energy-dissipation inequality for the generalized solution $u\in \mathcal C_w$ of Lemma~\ref{lem:conv}. Let us define for every $t\in [0,T]$ the total energy $\mathcal E(t)$ and the dissipation $\mathcal D(t)$ as
\begin{align*}
&\mathcal{E}(t):=\frac{1}{2}\|\dot u(t)\|_H^2+\frac{1}{2}(\C e u(t),e u(t))_{H}+\frac{1}{2}(\G(t)(e u(t)-eu^0),e u(t)-eu^0)_{H}\nonumber\\
&\hspace{5.5cm}-\frac{1}{2}\int_0^t(\dot \G(t-r)(eu(t)-e u(r)),eu(t)-e u(r))_{H}\,\de r,\\
&\mathcal D(t):=-\frac{1}{2}\int_0^t(\dot \G(r)(e u(r)-eu^0),e u(r)-eu^0)_{H}\,\de r\nonumber\\
&\hspace{5.5cm}+\frac{1}{2}\int_0^t\int_0^r(\ddot \G(r-s)(e u(r)-e u(s)),e u(r)-e u(s))_{H}\,\de s\,\de r.
\end{align*}
Notice that $\mathcal {E}(t)$ is well defined for every time $t\in[0, T]$ since $u\in C_w^0([0, T];U_T)$ and $\dot u\in C_w^0([0, T];H)$. Moreover, by the initial conditions we have
$$\mathcal E(0)=\frac{1}{2}\|u^1\|_H^2+\frac{1}{2}(\C e u^0,e u^0)_{H}.$$

\begin{proposition}\label{prop:reg_enin}
Assume~\eqref{eq:data2}--\eqref{eq:C} and~\eqref{eq:G1}--\eqref{eq:data1}. Then the generalized solution $u\in \mathcal C_w$ to system~\eqref{eq:Gsystem}  of Lemma~\ref{lem:conv} satisfies for every $t\in[0,T]$ the following energy-dissipation inequality 
\begin{equation}\label{eq:enin}
\mathcal{E}(t)+\mathcal D(t)\leq \mathcal{E}(0)+\mathcal{W}_{tot}(t), 
\end{equation}
where the total work is defined as
\begin{align}\label{tot-work}
\mathcal W_{tot}(t):=&\int_0^t [(f(r),\dot u(r)-\dot z(r))_H-(\dot N(r),u(r)-z(r))_{H_N}-(\dot u(r),\ddot z(r))_H+(\C eu(r),e\dot z(r))_H]\,\de r\nonumber\\
& +(N(t),u(t)-z(t))_{H_N}-(N(0),u^0-z(0))_{H_N}+(\dot{u}(t),\dot{z}(t))_H -(u^1,\dot{z}(0))_H\nonumber\\
& +\int_0^t(\G(r)(e u(r)-eu^0),e \dot z(r))_{H}\,\de r+\int_0^t\int_0^r(\dot \G(r-s)(e u(s)-e u(r)),e \dot z(r))_{H}\,\de s\,\de r.
 \end{align}
\end{proposition}

\begin{proof}
Fixed $t\in (0, T]$ and $n\in\mathbb N$ there exists a unique $i=i(n)\in\{1,\dots,n\}$ such that $t\in((i-1)\tau_n,i\tau_n]$. In particular, $i(n)=\left\lceil\frac{t}{\tau_n}\right\rceil.$ After setting $t_n:=i\tau_n$ and using that $\delta \G_n^0=0$, we rewrite~\eqref{d-en-bal-scomp} as
\begin{align}\label{en-bal-scomp}
&\frac{1}{2}\|\tilde{u}_n^+(t)\|_H^2+\frac{1}{2}(\C e u_n^+(t),e u_n^+(t))_{H}+\frac{1}{2}(\G_n^+(t) (e u_n^+(t)-eu^0),e u_n^+(t)-eu^0)_{H}\nonumber\\
&\quad-\frac{1}{2}\int_0^{t_n}(\tilde\G_n^+(t_n-r)(e u_n^+(t)-e u_n^+(r)),e u_n^+(t)-e u_n^+(r))_{H}\,\de r\nonumber\\
&\quad+\frac{1}{2}\int_{\tau_n}^{t_n}\int_0^{r_n-\tau_n}(\dot{\tilde{\G}}_n(r_n-s)(e u_n^+(r)-e u_n^+(s)),e u_n^+(r)-e u_n^+(s))_{H}\,\de s\,\de r\nonumber\\
&\quad-\frac{1}{2}\int_0^{t_n}(\tilde\G_n^+(r)(e u_n^+(r)-eu^0),e u_n^+(r)-eu^0)_{H}\,\de r\leq\frac{1}{2}\|u^1\|_H^2+\frac{1}{2}(\C e u^0,e u^0)_{H}+\mathcal W_n^+(t),
\end{align}
where $r_n:=\left\lceil \frac{r}{\tau_n}\right\rceil\tau_n$ for $r\in(\tau_n,t_n)$, and the approximate total work $\mathcal W^+_n(t)$ is given by
\begin{align*}
\mathcal W^+_n(t):=&\int_0^{t_n}[(f_n^+(r),\tilde u_n^+(r)-\tilde z_n^+(r))_H+(N_n^+(r),\tilde u_n^+(r)-\tilde z_n^+(r))_{H_N}+(\dot{\tilde u}_n(r),{\tilde z}^+_n(r))_H]\,\de r\\
&+\int_0^{t_n}[(\mathbb Ceu_n^+(r),e\tilde z_n^+(r))_H+(\G_n^-(r)(eu_n^+(r)-eu^0),e\tilde{z}_n^+(r))_H]\,\de r\\
&+\int_{\tau_n}^{t_n}\int_0^{r_n-\tau_n}(\tilde{\G}_n^-(r_n-s)(eu_n^+(s)-eu_n^+(r)),e\tilde z_n^+(r))_H\,\de s\,\de r.
\end{align*}
By~\eqref{eq:C},~\eqref{eq:G2}, and~\eqref{point-conv} we derive
\begin{align}\label{lsc-first}
\| \dot{u}(t)\|^2_H&\leq\liminf_{n\to\infty}\| \tilde{u}_n^+(t)\|^2_H,\\
(\C eu(t),eu(t))_H&\leq \liminf_{n\to\infty} (\C eu_n^+(t),eu_n^+(t))_H,\\
(\G(t)(eu(t)-eu^0),eu(t)-eu^0)_H&\leq \liminf_{n\to\infty} (\G(t)(eu_n^+(t)-eu^0),eu_n^+(t)-eu^0)_H.\label{lsc-2}
\end{align}
Moreover, the estimate~\eqref{d-eq:est} imply
\begin{align*}
\left| ((\G(t)-\G_n^+(t))(eu_n^+(t)-eu^0),eu_n^+(t)-eu^0)_H\right|
% &\leq\|\dot{\G}\|_{C^0([0,T];B)}\|eu_n^+(t)-eu^0\|_H^2(t_n-t)
\leq 4C^2\|\dot \G\|_{C^0([0,T];B)} \tau_n\xrightarrow[n\to \infty]{}0,
\end{align*}
which together with inequality~\eqref{lsc-2} gives
\begin{equation}
(\G(t)(e u(t)-eu^0),e u(t)-eu^0)_H \leq \liminf_{n\to\infty}(\G_n^+(t)(e u_n^+(t)-eu^0),e u_n^+(t)-eu^0)_H.
\end{equation}
By~\eqref{eq:G3} and~\eqref{point-conv}, for every $r\in(0,t)$ we have
\begin{align*}
(-\dot\G(t-r)(eu(t)-eu(r)),eu(t)-eu(r))_H\le\liminf_{n\to\infty}(-\dot\G(t-r)(eu_n^+(t)-eu_n^+(r)),eu_n^+(t)-eu_n^+(r))_H.
\end{align*}
Moreover
\begin{equation*}
\|\tilde\G_n^+(t_n-r)-\dot\G(t-r)\|_B\le\dashint_{t_n-r_n}^{t_n-r_n+\tau_n}\|\dot \G(s)-\dot\G(t-r)\|_B\,\de s\xrightarrow[n\to\infty]{} 0
\end{equation*}
because $t_n-r_n\to t-r$. Hence, we can argue as before to deduce
\begin{align*}
(-\dot\G(t-r)(eu(t)-eu(r))&,eu(t)-eu(r))_H\\
&\le\liminf_{n\to\infty}(-\tilde\G_n^+(t_n-r)(eu_n^+(t)-eu_n^+(r)),eu_n^+(t)-eu_n^+(r))_H.
\end{align*}
In particular, we can use Fatou's lemma and the fact that $t\le t_n$ to obtain
\begin{align*}
\int_0^t(-\dot\G(t-r)(eu(t)-eu(r))&,eu(t)-eu(r))_H\,\de r\\
&\le\liminf_{n\to\infty}\int_0^{t_n}(-\tilde\G_n^+(t_n-r)(eu_n^+(t)-eu_n^+(r)),eu_n^+(t)-eu_n^+(r))_H\,\de r.
\end{align*}
By arguing in a similar way, we can derive 
\begin{align*}
&\int_0^t(-\dot\G(r)(eu(r)-eu^0),eu(r)-eu^0)_H\,\de r\le\liminf_{n\to\infty}\int_0^{t_n}(-\tilde\G_n^+(r)(eu_n^+(r)-eu^0),eu_n^+(r)-eu^0)_H\,\de r.
\end{align*}

Let us consider the double integral in the left-hand side. We fix $r\in(0,t)$ and by~\eqref{eq:G4} for every $s\in(0,r)$ we have 
\begin{align*}
(\ddot\G(r-s)(eu(r)-eu(s))&,eu(r)-eu(s))_H\\
&\le\liminf_{n\to\infty}(\ddot\G(r-s)(eu_n^+(r)-eu_n^+(s)),eu_n^+(r)-eu_n^+(s))_H.
\end{align*}
Moreover, for a.e. $s\in (0,r_n-\tau_n)$ by defining $s_n:=\left\lceil \frac{s}{\tau_n}\right\rceil\tau_n$ we deduce
\begin{equation*}
\|\dot{\tilde\G}_n(r_n-s)-\ddot\G(r-s)\|_B\le\dashint_{r_n-s_n}^{r_n-s_n+\tau_n}\dashint_{\lambda-\tau_n}^\lambda\|\ddot \G(\theta)-\ddot\G(r-s)\|_B\,\de\theta\,\de \lambda\xrightarrow[n\to\infty]{} 0.
\end{equation*}
Therefore, for a.e. $s\in(0,r)$ we get
\begin{align*}
(\ddot\G(r-s)(eu(r)-eu(s))&,eu(r)-eu(s))_H\\
&\le\liminf_{n\to\infty}(\dot{\tilde\G}_n(r_n-s)(eu_n^+(r)-eu_n^+(s)),eu_n^+(r)-eu_n^+(s))_H,
\end{align*}
since $s\in (0,r_n-\tau_n)$ for $n$ large enough. If we apply again Fatou's lemma we conclude
\begin{align*}
\int_0^r(\ddot\G(r-s)(eu(r)-eu(s))&,eu(r)-eu(s))_H\,\de s\\
&\le\liminf_{n\to\infty}\int_0^r(\dot{\tilde\G}_n(r_n-s)(eu_n^+(r)-eu_n^+(s)),eu_n^+(r)-eu_n^+(s))_H\,\de s.
\end{align*}
By~\eqref{d-eq:est} we get
\begin{align*}
\left|\int_{r_n-\tau_n}^r(\dot{\tilde\G}_n(r_n-s)(eu_n^+(r)-eu_n^+(s)),eu_n^+(r)-eu_n^+(s))_H\,\de s\right|\le 4C^2\|\ddot \G\|_{C^0([0,T];B)}(r-r_n+\tau_n)\xrightarrow[n\to\infty]{}0,
\end{align*}
from which we derive
\begin{align*}
\int_0^r(\ddot\G(r-s)(eu(r)-eu(s))&,eu(r)-eu(s))_H\,\de s\\
&\le\liminf_{n\to\infty}\int_0^{r_n-\tau_n}(\dot{\tilde\G}_n(r_n-s)(eu_n^+(r)-eu_n^+(s)),eu_n^+(r)-eu_n^+(s))_H\,\de s.
\end{align*}
Since this is true for every $r\in(0,t)$, arguing as before we obtain
\begin{align*}
\int_0^t\int_0^r(\ddot\G(r-s)(eu(r)-eu&(s)),eu(r)-eu(s))_H\,\de s\,\de r\\
&\le\liminf_{n\to\infty}\int_{\tau_n}^{t_n}\int_0^{r_n-\tau_n}(\dot{\tilde\G}_n(r_n-s)(eu_n^+(r)-eu_n^+(s)),eu_n^+(r)-eu_n^+(s))_H\,\de s\,\de r.
\end{align*}

Let us study the right-hand side of~\eqref{en-bal-scomp}. Given that
\begin{align*}
\chi_{[0,t_n]}f^+_n&\xrightarrow[n\to\infty]{L^2(0, T;H)}\chi_{[0,t]}f, & \tilde{u}^+_n-\tilde{z}^+_n&\xrightharpoonup[n\to\infty]{L^2(0, T;H)}\dot{u}-\dot z, \\
\chi_{[0,t_n]}\G_n^-e\tilde{z}^+_n&\xrightarrow[n\to\infty]{L^1(0,T;H)}\chi_{[0,t]}\G e\dot{z}, &u_n^+&\xstararrow{n\to\infty}{L^\infty(0,T;U_T)}u,
\end{align*}
we can deduce
\begin{align}
\int_0^{t_n}(f^+_n(r),\tilde{u}^+_n(r)-\tilde{z}^+_n(r))_H\,\de r &\xrightarrow[n\to\infty]{} \int_0^t(f(r),\dot{u}(r)-\dot z(r))_H\,\de r,\\
\int_0^{t_n}(\C eu_n^+(r),e\tilde{z}_n^+(r))_H\,\de r &\xrightarrow[n\to\infty]{} \int_0^t (\C eu(r),e\dot z(r))_H\,\de r,\\
\int_0^{t_n}(\G_n^-(r)(eu_n^+(r)-eu^0),e\tilde{z}_n^+(r))_H\,\de r &\xrightarrow[n\to\infty]{} \int_0^t (\G(r) (eu(r)-eu^0),e\dot z(r))_H\,\de r .
\end{align}
By using the same argumentations of~\eqref{arg-1}--\eqref{arg-3}, together with the dominate convergence theorem, we can write
\begin{align}
 \int_{\tau_n}^{t_n}\int_0^{r_n-\tau_n}(\tilde{\G}_n^-(r_n-s)(eu_n^+(s)-eu_n^+(r))&,e\tilde z_n^+(r))_H\,\de s \,\de r\nonumber\\
 &\xrightarrow[n\to\infty]{}\int_0^{t}\int_0^r(\dot{\G}(r-s)(eu(s)-eu(r)),e\dot z(r))_H\,\de s\,\de r.
\end{align} 
Thanks to the discrete integration by parts formulas~\eqref{dis-part-1}--\eqref{dis-part-3} we have
\begin{align*}
& \int_0^{t_n}(\dot{\tilde u}_n(r),\tilde z_n^+(r))_H\,\de r = (\tilde u_n^+(t),\tilde z_n^+(t))_H -(u^1,\dot z(0))_H- \int_0^{t_n}(\tilde u_n^-(r),\dot{\tilde z}_n(r))_H\,\de r ,\\
& \int_0^{t_n}(N_n^+(r),\tilde u_n^+(r)-\tilde z_n^+(r))_{H_N}\,\de r = (N_n^+(t),u_n^+(t)- z_n^+(t))_{H_N}-(N(0),u^0- z(0))_{H_N}\\
&\hspace{8cm}- \int_0^{t_n}(\tilde N_n^+(r),u_n^-(r)- z_n^-(r))_{H_N}\,\de r.
\end{align*}
By arguing as before we deduce
\begin{align}
& \int_0^{t_n}(\dot{\tilde u}_n(r),\tilde z_n^+(r))_H\,\de r \xrightarrow[n\to \infty]{} (\dot u(t),\dot z(t))_H-(u^1,\dot z(0))_H- \int_0^{t}(\dot u(r),{\ddot z}(r))_H\,\de r ,\\
& \int_0^{t_n}(N_n^+(r),\tilde u_n^+(r)-\tilde z_n^+(r))_{H_N}\,\de r \nonumber\\
&\hspace{1cm}\xrightarrow[n\to \infty]{}(N(t),u(t)- z(t))_{H_N}-(N(0),u^0- z(0))_{H_N}- \int_0^{t}(\dot N(r),u(r)- z(r))_{H_N}\,\de r,\label{cont-last}
\end{align}
thanks to Lemma~\ref{lem:conv} and to the following convergences:
\begin{align*}
&\|\tilde{z}_n^+(t)-\dot{z}(t)\|_H
% =\left\|\frac{z(j\tau_n)-z((j-1)\tau_n)}{\tau_n}-\dot{z}(t)\right\|_H
\leq\dashint_{t_n-\tau_n}^{t_n}\|\dot{z}(r)-\dot{z}(t)\|_H\,\de r\xrightarrow[n \to \infty]{}0,\\
&\|z_n^+(t)-z(t)\|_{H_N}
% \leq C_{tr}\|z_n^+(t)-z(t)\|_{U_T}=C_{tr}\|z(j\tau_n)-z(t)\|_{U_T}
\leq C_{tr}\sqrt{\tau_n}\|\dot z\|_{L^2(0,T;U_0)}\xrightarrow[n \to \infty]{}0,\\
&\|N_n^+(t)-N(t)\|_{H_N}
% =\|N(j\tau_n)-N(t)\|_{H_N}
\leq \int_{t}^{t_n}\|\dot N(s)\|_{H_N}\,\de s\xrightarrow[n \to \infty]{}0,
\end{align*}
and 
\begin{align*}
\chi_{[0,t_n]}\dot{\tilde z}_n&\xrightarrow[n\to\infty]{L^1(0,T;H)}\chi_{[0,t]}\ddot{z}, & \tilde{u}_n^-&\xstararrow{n\to\infty}{L^\infty(0, T;H)}\dot u, \\
\chi_{[0,t_n]}{\tilde N}^+_n&\xrightarrow[n\to\infty]{L^1(0,T;H_N)}\chi_{[0,t]}\dot N, & u_n^--z_n^{-}&\xstararrow{n\to\infty}{L^\infty(0, T;U_T)}u-z.
\end{align*}
By combining~\eqref{en-bal-scomp} with~\eqref{lsc-first}--\eqref{cont-last} we deduce the energy-dissipation inequality~\eqref{eq:enin} for every $t\in(0,T]$. Finally, for $t=0$ the inequality trivially holds since $u(0)=u^0$ in $U_0$ and $\dot u(0)=u^1$ in $H$. 
\end{proof}

\begin{remark}
From the classical point of view, the total work on the solution $u$ at time $t\in[0, T]$ is given by
\begin{equation}\label{classic-work}
 \mathcal W^C_{tot}(t):=\mathcal W_{load}(t)+\mathcal W_{bdry}(t),
\end{equation}
where $\mathcal W_{load}(t)$ is the work on the solution $u$ at time $t\in[0, T]$ due to the loading term, which is defined as
\begin{equation*}
 \mathcal W_{load}(t):=\int_0^t(f(r), \dot u(r))_H\,\de r,
\end{equation*}
and $\mathcal W_{bdry}(t)$ is the work on the solution $u$ at time $t\in[0, T]$ due to the varying boundary conditions, which one expects to be equal to
\begin{align*}
\mathcal W_{bdry}(t):=\int_0^t(N(r),\dot u(r))_{H_N}\,\de r+\int_0^t(\C eu(r)\nu+\left(\frac{\de}{\de r}\int_0^r \G(r-s)(eu(s)-eu^0)\de s\right)\nu,\dot z(r))_{H_D}\,\de r.
\end{align*}
Unfortunately, $\mathcal W_{bdry}(t)$ is not well defined under our assumptions on $u$. In particular, the term involving the Dirichlet datum $z$ is difficult to handle since the trace of the function $\C eu(r)\nu+\frac{\de}{\de r}\left(\int_0^r \G(r-s)eu(s)\de s\right)\nu$ on $\partial_D\Omega$ is not well defined. If we assume that $u\in L^2(0,T;H^2(\Omega\setminus\Gamma;\R^d))\cap H^2(0,T;L^2(\Omega\setminus\Gamma;\R^d))$ and that $\Gamma$ is a smooth manifold, then the first term of $\mathcal W_{bdry}(t)$ makes sense and satisfies
\begin{align*}
\int_0^t(N(r),\dot u(r))_{H_N}\,\de r=(N(t), u(t))_{H_N}-(N(0), u(0))_{H_N}- \int_0^t(\dot N(r), u(r))_{H_N}\,\de r.
\end{align*}
Moreover, we have
\begin{align}\label{eq:Geu}
\frac{\de}{\de r}\int_0^r \G(r-s)(eu(s)-eu^0)\,\de s&=\G(0)(eu(r)-eu^0)+\int_0^r \dot{\G}(r-s)(eu(s)-eu^0)\,\de s\nonumber\\
&=\G(r)(eu(r)-eu^0)+\int_0^r \dot{\G}(r-s)(eu(s)-eu(r))\,\de s,
\end{align}
therefore $\left(\frac{\de}{\de r}\int_0^r \G(r-s)(eu(s)-eu^0)\de s\right)\nu\in L^2(0,T;H_D)$. By using~\eqref{eq:Gsystem}, together with the divergence theorem and the integration by parts formula, we derive
\begin{align}\label{eq:wbound}
&\int_0^t(\C eu(r)\nu+\left(\frac{\de}{\de r}\int_0^r \G(r-s)(eu(s)-eu^0)\,\de s\right)\nu,\dot z(r))_{H_D}\,\de r\nonumber\\
&=\int_0^t(\C eu(r)+\frac{\de}{\de r}\int_0^r \G(r-s)(eu(s)-eu^0)\,\de s,e \dot z(r))_H\de r\nonumber\\
&\quad+\int_0^t\left[(\div\left(\C eu(r)+\frac{\de}{\de r}\int_0^r \G(r-s)(eu(s)-eu^0)\,\de s\right),\dot z(r))_H-(N(r),\dot z(r))_{H_N}\right]\de r\nonumber\\
&=\int_0^t\left[(\C eu(r)+\frac{\de}{\de r}\int_0^r \G(r-s)(eu(s)-eu^0)\,\de s,e\dot z(r))_H+(\ddot u(r)-f(r),\dot z(r))_H-(N(r),\dot z(r))_{H_N}\right]\de r\nonumber\\
&=\int_0^{t}\left[(\C eu(r)+\frac{\de}{\de r}\int_0^r \G(r-s)(eu(s)-eu^0)\,\de s,e\dot z(r))_H-(f(r),\dot z(r))_H\right]\de r\nonumber\\
&\quad+\int_0^{t}\left[(\dot N(r),z(r))_{H_N}-(\dot u(r),\ddot z(r))_H\right]\de r+(\dot{u}(t),\dot{z}(t))_H -(u^1,\dot{z}(0))_H-(N(t),z(t))+(N(0),z(0)).
\end{align}
Therefore, by~\eqref{eq:Geu} and~\eqref{eq:wbound} we deduce the definition of total work given in~\eqref{tot-work} is coherent with the classical one~\eqref{classic-work}. 
\end{remark}

We conclude this subsection by showing that the generalized solution of Lemma~\ref{lem:conv} satisfies the initial conditions in a stronger sense than the ones stated in Definition~\ref{gen-sol}.

\begin{lemma}\label{in-data}
Assume~\eqref{eq:data2}--\eqref{eq:C} and~\eqref{eq:G1}--\eqref{eq:data1}. Then the generalized solution $u\in\mathcal C_w$ to system~\eqref{eq:Gsystem} of Lemma~\ref{lem:conv} satisfies
\begin{equation}\label{eq:incon}
\lim_{t\to 0^+}\|u(t)-u^0\|_{U_T}=0,\quad\lim_{t\to 0^+}\|\dot u(t)-u^1\|_H=0.
\end{equation}
In particular, the functions $u\colon [0, T]\to U_T$ and $\dot u\colon [0, T]\to H$ are continuous at $t=0$.
\end{lemma}

\begin{proof}
By sending $t\rightarrow0^+$ into the energy-dissipation inequality~\eqref{eq:enin} and using that $u\in C_w^0([0, T];U_T)$, $\dot u\in C_w^0([0, T];H)$, and the lower semicontinuity of the real functions 
$$t\mapsto \norm{\dot u(t)}^2_H,\quad t\mapsto (\C eu(t),eu(t))_H,$$
we deduce
% \begin{align*}
% \mathcal E(0)&=\frac{1}{2}\norm{u^1}_H^2+\frac{1}{2}(\C eu^0,eu^0)_H\le\frac{1}{2}\left[\liminf_{t\to 0^+}\norm{\dot u(t)}_H^2+\liminf_{t\to 0^+}(\C eu(t), eu(t))_H\right]\\
% &\leq \liminf_{t\to 0^+}\left[\frac{1}{2}\norm{\dot u(t)}_H^2+\frac{1}{2}(\C eu(t), eu(t))_H\right]=\liminf_{t\to 0^+}\mathcal{E}(t)\leq \limsup_{t\to 0^+}\mathcal{E}(t)\leq \mathcal E(0),
% \end{align*}
% Therefore, there exists $\lim_{t\to 0^+}\mathcal E(t)=\mathcal E(0)$. Moreover, we have
\begin{align*}
\mathcal E(0)&\le\frac{1}{2}\liminf_{t\to 0^+}\norm{\dot u(t)}_H^2+\frac{1}{2}\liminf_{t\to 0^+}(\C eu(t), eu(t))_H\nonumber\\
&\leq \frac{1}{2}\limsup_{t\to 0^+}\norm{\dot u(t)}_H^2+\frac{1}{2}\liminf_{t\to 0^+}(\C eu(t), eu(t))_H\leq\limsup_{t\to 0^+}\left[\frac{1}{2}\norm{\dot u(t)}_H^2+\frac{1}{2}(\C eu(t), eu(t))_H\right]\le \mathcal{E}(0)
\end{align*}
because the right-hand side of~\eqref{eq:enin} is continuous in $t$, $u(0)=u^0$ in $U_0$ and $\dot u(0)=u^1$ in $H$. This gives
\begin{equation*}
 \lim_{t\to 0^+}\norm{\dot u(t)}_H^2=\norm{u^1}_H^2,
\end{equation*}
and in a similar way, we can also obtain
\begin{equation*}
\lim_{t\to 0^+}(\C eu(t),eu(t))_H=(\C eu^0,eu^0)_H.
\end{equation*}
Since
$$\dot u(t)\xrightharpoonup[t\to 0^+]{H}u^1,\quad eu(t)\xrightharpoonup[t\to 0^+]{H}eu^0,$$ 
and $u\in C^0([0, T];H)$, we deduce~\eqref{eq:incon}.
\end{proof}

By combining the previous results together we obtain the following existence result for the system~\eqref{eq:Gsystem}.

\begin{theorem}\label{thm:reg_exis}
Assume~\eqref{eq:data2}--\eqref{eq:C} and~\eqref{eq:G1}--\eqref{eq:data1}. Then there exists a generalized solution $u\in \mathcal C_w$ to system~\eqref{eq:Gsystem}. Moreover, we have $u\in H^2(0,T;(U_0^D)')$ and it satisfies the energy-dissipation inequality~\eqref{eq:enin} and
\begin{equation*}
\lim_{t\to 0^+}\|u(t)- u^0\|_{U_T}=0,\quad\lim_{t\to 0^+}\|\dot u(t)-u^1\|_H=0.
\end{equation*}
\end{theorem}

\begin{proof}
It is enough to combine Lemma~\ref{lem:conv}, Corollary~\ref{in-con}, Lemma~\ref{lem:sol}, Proposition~\ref{prop:reg_enin}, and Lemma~\ref{in-data}.
\end{proof}

\subsection{Uniform energy estimates}

In this subsection we show that, under the stronger assumption~\eqref{eq:dataFz} on $z$, the generalized solution to~\eqref{eq:Gsystem} of Theorem~\ref{thm:reg_exis} satisfies some uniform estimates which depends on $\G$ only via $\|\G\|_{L^1(0,T;B)}$.

\begin{lemma}\label{lem:irr_est}
Assume~\eqref{eq:dataFz}--\eqref{eq:C} and~\eqref{eq:G1}--\eqref{eq:G4}. Let $u$ be the generalized solution to system~\eqref{eq:Gsystem} of Theorem~\ref{thm:reg_exis}. Then there exists a constant $M=M(z,N,f,u^0,u^1,\C,\|\G\|_{L^1(0,T;B)})$ such that
\begin{equation}\label{eq:L1Gest}
 \|\dot u(t)\|_H+\|eu(t)\|_H\le M\quad\text{for every $t\in[0,T]$}.
\end{equation}
\end{lemma}

\begin{proof}
We define
\begin{equation*}
K:=\sup_{t\in[0,T]}\|\dot u(t)\|_H=\|\dot u\|_{L^\infty(0,T;H)},\quad E:=\sup_{t \in[0,T]}\| e u(t)\|_{H}=\|eu\|_{L^\infty(0,T;H)}.
\end{equation*}
Notice that $K$ and $E$ are well-posed since $u\in C_w^0([0,T];U_T)$ and $\dot u\in C_w^0([0,T];H)$. Let us estimate the total work $\mathcal W_{tot}(t)$ in~\eqref{eq:enin} by means of $K$ and $E$.
Since
\begin{equation*}
 \|u(t)\|_{U_T}\le \|u^0\|_H+TK+E\quad\text{for every $t\in[0,T]$},
\end{equation*}
we have
\begin{align*}
\left|\int_0^t(f(r),\dot u(r))_H\,\de r \right|&\le \sqrt{T}\|f\|_{L^2(0,T;H)}K,\\
\left|\int_0^t(\dot N(r), u(r))_{H_N}\,\de r \right|&\le C_{tr}\|\dot N\|_{L^2(0,T;H_N)}\left(\|u^0\|_H+TK+E\right),\\
|(N(t),u(t))_{H_N}|&\le C_{tr}\|N\|_{C^0([0,T];H_N)}\left(\|u^0\|_H+TK+E\right),\\
|(N(0),u^0)_{H_N}|&\le C_{tr}\|N\|_{C^0([0,T];H_N)}\left(\|u^0\|_H+TK+E\right),\\
\left|\int_0^t(f(r),\dot z(r))_H\,\de r \right|&\le \sqrt{T}\|f\|_{L^2(0,T;H)}\|\dot z\|_{C^0([0,T];H)},\\
\left|\int_0^t(N(r), \dot z(r))_{H_N}\,\de r \right|&\le C_{tr}\|N\|_{C^0([0,T];H_N)}\|\dot z\|_{L^1(0,T;U_0)},\\
\left|\int_0^t(\C eu(r),e\dot z(r))_H\,\de r\right|&\le \|\C\|_B\|e\dot z\|_{L^1(0,T;H)}E,\\
\left|\int_0^t(\dot u(r),\ddot z(r))_H\,\de r \right|&\le \|\ddot z\|_{L^1(0,T;H)}K,\\
|(\dot u(t),\dot z(t))_H|&\le\|\dot z\|_{C^0([0,T];H)}K,\\
|(u^1,\dot z(0))_H|&\le \|\dot z\|_{C^0([0,T];H)}K.
\end{align*}
It remains to study the last two terms, which are
\begin{align*}
&\int_0^t(\G(r) (e u(r)-eu^0), e \dot z(r))_{H}\,\de r +\int_0^t\int_0^ r (\dot \G(r -s)(e u(s)- e u(r)), e \dot z(r))_{H}\,\de s\,\de r \\
&=\int_0^t(\G(0) (e u(r)-eu^0), e \dot z(r))_{H}\,\de r +\int_0^t\int_0^ r (\dot \G(r -s) (e u(s)-eu^0), e \dot z(r))_{H}\,\de s\,\de r .
\end{align*}
Since $z\in W^{2,1}(0,T;U_0)$, arguing as in Proposition~\ref{lem:equiv} we can deduce that the function
\begin{equation*}
p(t):=\int_0^t(\G(t- r)(e u(r)-eu^0), e \dot z(t))_{H}\,\de r
\end{equation*}
is absolutely continuous on $[0,T]$. In particular
\begin{equation*}
p(t)-p(0)=\int_0^t\dot p(r)\,\de r,
\end{equation*}
which gives
\begin{align}
&\int_0^t(\G(r) (e u(r)-eu^0), e \dot z(r))_{H}\,\de r +\int_0^t\int_0^ r (\dot \G(r -s) (e u(s)-eu(r)), e \dot z(r))_{H}\,\de s\,\de r \nonumber\\
&=\int_0^t(\G(t- r) (e u(r)-eu^0) , e \dot z(t))_{H}\,\de r-\int_0^t\int_0^ r(\G(r -s)(e u(s)-eu^0), e \ddot z(r))_{H}\,\de s\,\de r .\label{eq:totwork}
\end{align}
Hence, we deduce
\begin{align*}
&\left|\int_0^t(\G(r) (e u(r)-eu^0), e \dot z(r))_{H}\,\de r +\int_0^t\int_0^ r (\dot \G(r -s)(e u(s)- e u(r)), e \dot z(r))_{H}\,\de s\,\de r \right|\\
%&=\left|\int_0^t(\G(t- r) (e u(r)-eu^0) , e \dot z(t))_{H}\,\de r-\int_0^t\int_0^ r(\G(r -s) (e u(s)-eu^0), e \ddot z(r))_{H}\,\de s\,\de r \right|\\
&\hspace{5.3cm}\le 2(\|e\dot z\|_{C^0([0,T];H)}+\| e \ddot z\|_{L^1(0,T;H)})\|\G\|_{L^1(0,T;B)}E.
\end{align*}

Therefore, since
\begin{equation*}
\mathcal E(0)\le \frac{1}{2}\|u^1\|_H^2+\frac{1}{2}\|\C\|_B\|e u^0\|_H^2,
\end{equation*}
by~\eqref{eq:enin} we deduce the following estimate
\begin{align*}
&\|\dot u(t)\|_H^2+\gamma\| e u(t)\|_{H}^2\le C_0+C_1K+C_2E\quad\text{for every $t\in[0,T]$},
\end{align*}
where 
\begin{equation*}
C_0=C_0(z,N,f,u^0,u^1,\C),\quad C_1=C_1(f,z,N),\quad C_2=C_2(z,N,\C,\|\G\|_{L^1(0,T;B)}).
\end{equation*}
In particular, being the right-hand side independent of $t\in[0,T]$, we conclude
\begin{equation*}
K^2+\gamma E^2\le 2C_0+2C_1K+2C_2E\quad\text{for every $t\in[0,T]$}.
\end{equation*}
This implies the existence of a constant $M=M(C_0,C_1,C_2)$ for which~\eqref{eq:L1Gest} is satisfied.
\end{proof}

\begin{remark}
By the previous estimate, we can easily derive a uniform bound also for $\dot u$ in $H^1(0,T;(U_0^D)')$, which unfortunately depends on $\G$ via $\|\G(0)\|_B$. Indeed, let us assume that $z$, $N$, $f$, $u^0$, $u^1$, $\C$, and $\G$ satisfy~\eqref{eq:dataFz}--\eqref{eq:C} and~\eqref{eq:G1}--\eqref{eq:G4} and let $u$ be the generalized solution of Theorem~\ref{thm:reg_exis}. Thanks to~\eqref{eq:enin} and~\eqref{eq:L1Gest} there exists a constant $\overline M=\overline M(z,N,f,u^0,u^1,\C,\|\G\|_{L^1(0;T;B)})$ such that for every $t\in[0,T]$
\begin{equation*}
\|eu(t)\|_H^2+(\G(t)(e u(t)-eu^0),eu(t)-eu^0)_{H}+\int_0^t(-\dot \G(t- r)(e u(t)- e u(r)), e u(t)- e u(r))_{H}\,\de r \le \overline M.
\end{equation*}
By equation~\eqref{eq:Ggen} it is easy to see that $\dot u\in H^1(0,T;(U_0^D)')$ and that $\ddot u$ satisfies for a.e. $t\in(0,T)$ and for every $v\in U_0^D$
\begin{align*}
|\langle \ddot u(t),v\rangle_{(U^D_0)'}|&\le \|\C\|_B\|eu(t)\|_H\|ev\|_H+\sqrt{(\G(t)(eu(t)-eu^0), e u(t)-eu^0)_{H}}\sqrt{(\G(t) e v, e v)_{H}}\\
&\quad+\sqrt{\int_0^t(-\dot\G(t- r)(eu(t)-eu(r)), eu(t)-eu(r))_{H}\,\de r }\sqrt{\int_0^t(-\dot \G(t- r) e v, e v)_{H}\,\de r }\\
&\quad+\|f(t)\|_H\|v\|_H+\|N(t)\|_{H_N}\|v\|_{H_N}.
\end{align*}
Hence, we derive
\begin{align*}
|\langle \ddot u(t),v\rangle_{(U^D_0)'}|^2
% &\le 5\|\C\|_B^2\|eu(t)\|^2_H\|ev\|_H^2+5(\G(t) (e u(t)-eu^0), e u(t)-eu^0)_{H}(\G(t) e v, e v)_{H}\\
% &\quad+5\int_0^t(-\dot\G(t- r)(eu(t)-eu(r)), eu(t)-eu(r))_{H}\,\de r\int_0^t(-\dot \G(t- r) e v, e v)_{H}\,\de r \\
% &\quad+5\|f(t)\|_H^2\|v\|_H^2+5\|N(t)\|_{H_N}^2\|v\|_{H_N}^2\\
&\le 5\overline M\|ev\|_H^2+5\overline M(\G(t) e v, e v)_{H}+5\overline M\int_0^t(-\dot \G(t- r) e v, e v)_{H}\,\de r\\
&\quad+5\|f(t)\|_H^2\|v\|_H^2+5C^2_{tr}\|N(t)\|_{H_N}^2\|v\|_{U_0}^2\\
&=5\overline M\|ev\|_H^2+5\overline M(\G(0) e v, e v)_{H}+5\|f(t)\|_H^2\|v\|_H^2+5C^2_{tr}\|N(t)\|_{H_N}^2\|v\|_{U_0}^2,
\end{align*}
which gives
% \begin{equation*}
% \|\ddot u(t)\|_{(U^D_0)'}^2\le 5\overline M+5\overline M\|\G(0)\|_B+5\|f(t)\|_H^2+5C^2_{tr}\|N(t)\|_{H_N}^2.
% \end{equation*}
% Thus, we get
\begin{equation*}
\|\ddot u\|_{L^2(0,T;(U_0^D)')}^2\le 5\overline MT+5\overline MT \|\G(0)\|_B+5\|f\|_{L^2(0,T;H)}^2+5C^2_{tr}\|N\|_{L^2(0,T;H_N)}^2.
\end{equation*}
Therefore the bounds on $\ddot u$ depends on $\|\G(0)\|_B$ even when $z\in W^{2,1}(0,T;U_0)$. 
\end{remark}

As explained in the previous remark, we can not deduce a uniform bound for $\dot u$ in $H^1(0,T;(U_0^D)')$ depending on $\G$ only via its $L^1$-norm. On the other hand, the bound on $\dot u$ in $H^1(0,T;(U_0^D)')$ is useful if we want to prove the existence of a generalized solution $u^*$ to the fractional Kelvin-Voigt system~\eqref{eq:Fsystem}, especially to show that $\dot u^*\in C_w^0([0,T];H)$. To overcome this problem, we introduce another function that is related to $\dot u$ and for which is possible to derive a uniform bound. Let us consider the auxiliary function $\alpha\colon [0,T]\to (U_0^D)'$ defined as
\begin{align*}
\langle \alpha(t),v\rangle_{(U_0^D)'}:=(\dot u(t),v)_H+\int_0^t(\G(t-r) (e u(r)-eu^0), e v)_{H}\,\de r \quad\text{for every $v\in U_0^D$ and $t\in[0,T]$}. 
\end{align*}
Notice that $ \alpha \in C_w^0([0,T];(U_0^D)')$. Indeed, given $t^*\in[0,T]$ and $$\{t_k\}_k\subset[0,T] \quad \text{ such that } \quad t_k\xrightarrow[k\to\infty]{} t^*,$$
we have for every $v\in U_0^D$ the following convergence
\begin{align*}
\langle \alpha (t_k),v\rangle_{(U_0^D)'}&=(\dot u(t_k),v)_H+\int_0^{t_k}(\G(t_k- r) (e u(r)-eu^0), e v)_{H}\,\de r \\
&\xrightarrow[k\to\infty]{} (\dot u(t^*),v)_H+\int_0^{t^*}(\G(t^*- r) (e u(r)-eu^0), e v)_{H}\,\de r =\langle \alpha (t^*),v\rangle_{(U_0^D)'},
\end{align*}
since 
$$\dot u(t_k)\xrightharpoonup[k\to \infty]{H} \dot u(t^*),\quad \int_0^{t_k}(\G(t_k- r) (e u(r)-eu^0), e v)_{H}\,\de r\xrightarrow[k\to\infty]{} \int_0^{t^*}(\G(t^*- r) (e u(r)-eu^0), e v)_{H}\,\de r.$$ 
The second convergence is true because
\begin{align*}
\int_0^{t_k}(\G(t_k- r)& (e u(r)-eu^0), e v)_{H}\,\de r 
% &=\int_0^{t_k}(e u(r)-eu^0,\G(t_k- r) e v)_{H}\,\de r 
\\
&=\int_0^{t^*}(e u(r)-eu^0,\G(t_k- r) e v)_{H}\,\de r -\int_{t_k}^{t^*}(e u(r)-eu^0,\G(t_k- r) e v)_{H}\,\de r .
\end{align*}
Clearly 
$$\G(t_k-\,\cdot\,) e v\xrightarrow[k\to\infty]{L^1(0,t^*;H)} \G(t^*-\,\cdot\,) e v$$ 
while $ e u\in L^\infty(0,t^*;H)$. Therefore
\begin{align*}
\int_0^{t^*}(e u(r)-eu^0,\G(t_k- r) e v)_{H}\,\de r \xrightarrow[k\to \infty]{} &\int_0^{t^*}(e u(r)-eu^0,\G(t^*- r) e v)_{H}\,\de r\\ 
&=\int_0^{t^*}(\G(t^*- r) (e u(r)-eu^0) , e v)_{H}. 
\end{align*}
Moreover
\begin{align*}
\left|\int_{t_k}^{t^*}(e u(r)-eu^0,\G(t_k- r) e v)_{H}\,\de r \right|\le 2M\| e v\|_{H}\left|\int_0^{t_k-t^*}\|\G(r)\|_B\,\de r \right|\xrightarrow[k\to\infty]{} 0.
\end{align*}

For this function $\alpha$ is possible to find a uniform bound in $H^1(0,T;(U_0^D)')$ which depends on $\|\G\|_{L^1(0,T;B)}$. 

\begin{corollary}\label{coro:irr_est}
Assume~\eqref{eq:dataFz}--\eqref{eq:C} and~\eqref{eq:G1}--\eqref{eq:G4}. Then the function $\alpha\in H^1(0,T;(U_0^D)')$ and there exists a constant $\tilde M=\tilde M(z,N,f,u^0,u^1,\C,\|\G\|_{L^1(0,T;B)})$ such that
\begin{equation}\label{eq:alpha}
 \|\alpha\|_{H^1(0,T;(U_0^D)')}\le \tilde M.
\end{equation}
\end{corollary}

\begin{proof}
First, by Lemma~\ref{lem:irr_est} we have
\begin{equation*}
\| \alpha (t)\|_{(U_0^D)'}\le M(1+2\|\G\|_{L^1(0,T;B)})\quad\text{for every $t\in[0,T]$}.
\end{equation*}
Moreover, by the definition of generalized solution, we deduce that for every $\psi\in C_c^1(0,T)$ and $v\in U_0^D$ it holds
\begin{equation*}
-\int_0^T\langle \alpha (t),v\rangle_{(U_0^D)'}\dot \psi(t)\,\de t=-\int_0^T(\C eu(t),ev)_H\psi(t)\,\de t+\int_0^T(f(t),v)_H\psi(t)\,\de t+\int_0^T(N(t),v)_{H_N}\psi(t)\,\de t.
\end{equation*}
This gives that there exists $\dot \alpha \in L^2(0,T;(U_0^D)')$ and 
\begin{equation*}
\langle \dot \alpha (t),v\rangle_{(U_0^D)'}=-(\C eu(t),ev)_H+(f(t),v)_H+(N(t),v)_{H_N}\quad\text{for every $v\in U_0^D$ and for a.e. $t\in(0,T)$}. 
\end{equation*}
In particular, $ \alpha \in C^0([0,T];(U_0^D)')$ and
\begin{equation*}
 \|\dot \alpha \|_{L^2(0,T;(U_0^D)')}^2\le 3M^2T\|\C\|_B^2+ 3\|f\|_{L^2(0,T;H)}^2+3C_{tr}^2\|N\|_{L^2(0,T;H_N)}^2,
\end{equation*}
which gives~\eqref{eq:alpha}.
\end{proof}

%---------------------------
% The case $\G$ unbounded
%---------------------------

\section{The fractional Kelvin-Voigt's model}\label{sec:irr}

In this section we prove the existence of a generalized solution to~\eqref{eq:Fsystem} for a tensor $\F$ which is not necessary bounded at $t=0$, as it happens in~\eqref{eq:Fdef}. Here, we assume that our data $z,N,f,u^0,u^1,\C$, and $\F$ satisfy the conditions~\eqref{eq:dataFz}--\eqref{eq:F4}. To prove the existence of a generalized solution to~\eqref{eq:Fsystem} under these assumptions, we first regularize $\F$ by a parameter $\epsilon>0$ and we consider system~\eqref{eq:Gsystem} associated to this regularization. Then, we take the solution $u^\epsilon$ given by Theorem~\ref{thm:reg_exis} and thanks to Lemma~\ref{lem:irr_est} and Corollary~\ref{coro:irr_est} we obtain a generalized solution to~\eqref{eq:Fsystem}.

Let us regularize $\F$ by defining
\begin{equation*}
\G^{\epsilon}(t):=\F\left(t+\epsilon\right)\quad\text{for $t\in[0,T]$ and $\epsilon\in(0,\delta_0)$}.
\end{equation*}
Clearly $\G^{\epsilon}$ satisfies~\eqref{eq:G1}--\eqref{eq:G4}. Moreover, we have $\G^{\epsilon}\to \F$ in $L^1(0,T;B)$ since $\F\in L^1(0,T+\delta_0;B)$. For every fixed $\epsilon\in(0,\delta_0)$ we can consider the generalized solution $u^\epsilon$ to system~\eqref{eq:Gsystem} with $\G$ replaced by $\G^\epsilon$ of Theorem~\ref{thm:reg_exis}. By Lemma~\ref{lem:irr_est} and Corollary~\ref{coro:irr_est} we deduce the following compactness result:

\begin{lemma}\label{lem:ir-con}
Assume~\eqref{eq:dataFz}--\eqref{eq:F4}. For every $\epsilon\in (0,\delta_0)$ let $u^\epsilon$ be the generalized solution associated to system~\eqref{eq:Gsystem} with $\G$ replaced by $\G^\epsilon$ given by Theorem~\ref{thm:reg_exis}. Then there exists a function $u^*\in \mathcal C_w$ and a subsequence of $\epsilon$, not relabeled, such that
\begin{align}\label{ep-weak-conv}
& u^{\epsilon}\xrightharpoonup[\epsilon\to 0^+]{L^2(0,T;U_T)} u^*,\quad \dot u^{\epsilon}\xrightharpoonup[\epsilon\to 0^+]{L^2(0,T;H)} \dot u^*,
\end{align}
and for every $t\in[0,T]$
\begin{align}\label{ep-point-conv}
 & u^{\epsilon}(t)\xrightharpoonup[\epsilon\to 0^+]{U_T} u^*(t),\quad \dot u^{\epsilon}(t)\xrightharpoonup[\epsilon\to 0^+]{H} \dot u^*(t).
\end{align}
Moreover, $u^*(0)=u^0$ in $U_0$, $\dot u^*(0)=u^1$ in $H$, and $u^*(t)-z(t)\in U_t^D$ for every $t\in[0,T]$.
\end{lemma}

\begin{proof}
Thanks to Lemma~\ref{lem:irr_est} we deduce
\begin{equation*}
\|\dot u^{\epsilon}(t)\|_H+\| e u^{\epsilon}(t)\|_{H}\le M\quad\text{for every $t\in[0,T]$ and $\epsilon\in(0,\delta_0)$},
\end{equation*}
with a constant $M$ independent of $\epsilon$ since $\|\G^\epsilon\|_{L^1(0,T;B)}\le \|\F\|_{L^1(0,T+\delta_0;B)}$.
Hence, by Banach-Alaoglu's theorem and Lemma~\ref{lem:wc} there exists 
\begin{equation*}
 u^*\in C_w^0([0,T];U_T)\cap W^{1,\infty}(0,T;H)
\end{equation*}
and a not relabeled subsequence of $\epsilon$ such that 
\begin{align}\label{eq:ep-conv2}
& u^{\epsilon}\xrightharpoonup[\epsilon\to 0^+]{L^2(0,T;U_T)} u^*,\quad \dot u^{\epsilon}\xrightharpoonup[\epsilon\to 0^+]{L^2(0,T;H)} \dot u^*,\quad u^{\epsilon}(t)\xrightharpoonup[\epsilon\to 0^+]{U_T} u^*(t)\quad\text{for every $t\in[0,T]$}.
\end{align}
In particular, we deduce that $u^*(0)=u^0$ in $U_0$, $u^*(t)\in U_t$ and $u^*(t)-z(t)\in U_t^D$ for every $t\in[0,T]$. 

It remains to show that $\dot u^*\in C_w^0([0,T];H)$, $\dot u^*(0)=u^1$ in $H$, and that for every $t\in[0,T]$ 
$$\dot u^\epsilon(t)\xrightharpoonup[\epsilon\to 0^+]{H} \dot u^*(t).$$ 
To this aim we consider the auxiliary function defined at the end of the previous section. More precisely, for every $\epsilon\in(0,\delta_0)$ let $ \alpha ^{\epsilon}\colon [0,T]\to (U_0^D)'$ be defined as
\begin{align*}
\langle \alpha ^{\epsilon}(t),v\rangle_{(U_0^D)'}:=(\dot u^{\epsilon}(t),v)_H+\int_0^t(\G^{\epsilon}(t- r) (e u^{\epsilon}(r)-eu^0), e v)_{H}\,\de r \quad\text{for every $v\in U_0^D$ and $t\in[0,T]$}.
\end{align*}
In view of Corollary~\ref{coro:irr_est}, we have
\begin{equation*}
\| \alpha ^{\epsilon}\|_{H^1(0,T;(U_0^D)')}\le \tilde M\quad\text{for every $\epsilon\in(0,\delta_0)$},
\end{equation*}
with $\tilde M$ independent of ${\epsilon}>0$ being $\|\G^\epsilon\|_{L^1(0,T;B)}\le \|\F\|_{L^1(0,T+\delta_0;B)}$. Hence, up to extract a further subsequence, there exists $ \alpha ^*\in H^1(0,T;(U_0^D)')$ such that
\begin{align}\label{eq:epconv3}
\alpha ^{\epsilon}\xrightharpoonup[\epsilon\to 0^+]{H^1(0,T;(U_0^D)')} \alpha ^*,\qquad \alpha ^{\epsilon}(t)\xrightharpoonup[\epsilon\to 0^+]{(U_0^D)'} \alpha ^*(t)\quad\text{for every $t\in[0,T]$}.
\end{align}
In particular, since $ \alpha ^{\epsilon}(0)=u^1$ in $(U_0^D)'$ we conclude that $ \alpha^* (0)=u^1$ in $(U_0^D)'$. We claim
\begin{equation*}
\langle \alpha ^*(t),v\rangle_{(U_0^D)'}=(\dot u^*(t),v)_H+\int_0^t(\F(t- r) (e u^*(r)-eu^0), e v)_{H}\,\de r \quad\text{for every $v\in U_0^D$ and for a.e. $t\in(0,T)$}.
\end{equation*}
Indeed, for every $\varphi\in C^\infty_c(0,T;U_0^D)$ we have
\begin{align*}
\int_0^T\langle \alpha ^{\epsilon}(t),\varphi(t)\rangle_{(U_0^D)'}\,\de t&=\int_0^T(\dot u^{\epsilon}(t),\varphi(t))_H\,\de t+\int_0^T\int_0^t(\G^{\epsilon}(t- r) (e u^{\epsilon}(r)-eu^0), e \varphi(t))_{H}\,\de r \,\de t\\
&\xrightarrow[\epsilon\to 0^+]{} \int_0^T(\dot u^*(t),\varphi(t))_H\,\de t+\int_0^T\int_0^t(\F(t- r) (e u^*(r)-eu^0), e \varphi(t))_{H}\,\de r \,\de t.
\end{align*}
Notice that this convergence is true thanks to~\eqref{eq:ep-conv2} and
$$\G^{\epsilon}(t-\,\cdot\,)\xrightarrow[\epsilon\to 0^+]{L^1(0,t;B)} \F(t-\,\cdot\,),$$
which gives
\begin{align*}
 \int_0^T(\dot u^{\epsilon}(t),\varphi(t))_H\,\de t&\xrightarrow[\epsilon\to 0^+]{} \int_0^T(\dot u^*(t),\varphi(t))_H\,\de t,\\
 \int_0^t(\G^{\epsilon}(t- r) (e u^{\epsilon}(r)-eu^0),e\varphi(t))_H\,\de r &\xrightarrow[\epsilon\to 0^+]{} \int_0^t(\F(t- r) (e u^*(r)-eu^0),e\varphi(t))_H\,\de r .
\end{align*}
Hence by the dominated convergence theorem we have
\begin{align*}
\int_0^T\int_0^t(\G^{\epsilon}(t- r) (e u^{\epsilon}(r)-eu^0), e \varphi(t))_{H}\,\de r \,\de t\xrightarrow[\epsilon\to 0^+]{}\int_0^T\int_0^t(\F(t- r) (e u^*(r)-eu^0), e \varphi(t))_{H}\,\de r \,\de t.
\end{align*}
Therefore, for a.e. $t\in(0,T)$ we deduce
\begin{align*}
\langle \dot u^*(t),v\rangle_{(U_0^D)'}&=(\dot u^*(t),v)_H=\langle \alpha ^*(t),v\rangle_{(U_0^D)'}-\int_0^t(\F(t- r) (e u^*(r)-eu^0), e v)_{H}\,\de r \quad\text{for every $v\in U_0^D$}. 
\end{align*}
Notice the function on the right-hand side is well defined in $(U_0^D)'$ for every $t\in[0,T]$. Therefore, we can extend $\dot u^*$ to a function defined in the whole interval $[0,T]$ with values in $(U_0^D)'$. In particular, we deduce $\dot u^*\in C_w^0([0,T];(U_0^D)')$, arguing in a similar way as we did in the previous section for $ \alpha $, and thanks to the fact that $\dot u^*(0)= \alpha^*(0)=u^1$ in $(U_0^D)'$. Therefore, since $\dot u^*\in C_w^0([0,T];(U_0^D)')\cap L^\infty(0,T;H)$ we derive that $\dot u^*\in C_w^0([0,T];H)$ (thanks to Lemma~\ref{lem:wc}), and that $\dot u^*(0)=u^1$ in $H$. Finally, we have
\begin{equation}\label{eq:ue}
\dot u^\epsilon(t)\xrightharpoonup[\epsilon\to 0^+]{(U_0^D)'}\dot u^*(t)\quad\text{for every $t\in[0,T]$}
\end{equation}
by definition of $\dot u^*$ and by~\eqref{eq:ep-conv2} and~\eqref{eq:epconv3}. The convergence~\eqref{eq:ue} combined with 
$$\|\dot u^\epsilon(t)\|_H\le M\quad \text{for every $t\in[0,T],$}$$ give us the last convergence required.
\end{proof}

We can now prove the main existence result of Theorem~\ref{thm:irr_exis} for the fractional Kelvin-Voigt's system involving Caputo's derivative.

\begin{proof}[Proof of Theorem~\ref{thm:irr_exis}]
It is enough to show that the function $u^*$ given by Lemma~\ref{lem:ir-con} is a generalized solution to~\eqref{eq:Fsystem}. To this aim, it remains to prove that $u^*$ satisfies~\eqref{eq:Fgen}. For every $\varphi\in \Cc$ we know that the function $u^\epsilon \in\mathcal C_w$ satisfy for every $\epsilon\in(0,\delta_0)$ the following equality
\begin{align*}
&-\int_0^T(\du (t),\dot \varphi(t))_H\,\de t+\int_0^T(\C e\ue(t),e\varphi(t))_H\,\de t-\int_0^T\int_0^t(\G^\epsilon(t-r)(e \ue(r)-eu^0),e \dot \varphi(t))_{H}\,\de r\,\de t\\
&=\int_0^T(f(t),\varphi(t))_H\,\de t+\int_0^T(N(t),\varphi(t))_{H_N}\,\de t.
\end{align*}
Let us pass to the limit as $\epsilon\to 0^+$. Clearly, by~\eqref{ep-weak-conv} we have
\begin{align*}
\int_0^T(\dot u^\epsilon(t),\dot \varphi(t))_H\,\de t&\xrightarrow[\epsilon \to 0^+]{}\int_0^T(\dot u^*(t),\dot \varphi(t))_H\,\de t,\\
\int_0^T(\C e\ue(t),e\varphi(t))_H\,\de t &\xrightarrow[\epsilon \to 0^+]{}\int_0^T(\C e u^*(t),e\varphi(t))_H\,\de t.
\end{align*}
It remains to study the behaviour as $\epsilon\to 0^+$ of 
\begin{equation*}
\int_0^T\int_0^t(\G^\epsilon(t-r)(e u^\epsilon(r)-eu^0),e \dot \varphi(t))_{H}\,\de r\,\de t.
\end{equation*}
We define for every $\epsilon\in(0,\delta_0)$ the function
\begin{equation*}
v^{\epsilon}(t):=\int_0^t(\G^{\epsilon}(t- r)-\F(t- r))(e u^{\epsilon}(r)-eu^0)\,\de r \quad\text{for $t\in[0,T]$}.
\end{equation*}
By~\eqref{eq:L1Gest} for every $t\in[0,T]$ it holds
\begin{equation}\label{eps-bound-eu}
\|v^\epsilon(t)\|_H\le \|\G^\epsilon-\F\|_{L^1(0,T;B)}\|eu^\epsilon-eu^0\|_{L^\infty(0,T;H)}\le 2M\|\G^\epsilon-\F\|_{L^1(0,T;B)},
\end{equation}
with $M$ independent of $\epsilon$ being $\|\G^\epsilon\|_{L^1(0,T;B)}\le \|\F\|_{L^1(0,T+\delta_0;B)}$. Notice that
\begin{align*}
\int_0^T\int_0^t(\G^{\epsilon}(t- r)& (e u^{\epsilon}(r)-eu^0), e \dot\varphi(t))_{H}\,\de r \,\de t\\
&=\int_0^T(v^{\epsilon}(t), e \dot\varphi(t))_{H}\,\de t+\int_0^T\int_0^t(\F(t- r)(e u^{\epsilon}(r)-eu^0), e \dot\varphi(t))_{H}\,\de r \,\de t,
\end{align*}
and thanks to~\eqref{eps-bound-eu} and to the fact that $\G^\epsilon\to \F$ in $L^1(0,T;B)$ as $\epsilon\to 0^+$, we get
\begin{align*}
\left|\int_0^T(v^{\epsilon}(t), e \dot\varphi(t))_{H}\,\de t\right|&\le \int_0^T\|v^{\epsilon}(t)\|_{H}\| e \dot \varphi(t)\|_H\,\de t\le 2 M\|\G^{\epsilon}-\F\|_{L^1(0,T;B)}\| e \dot \varphi\|_{L^1(0,T;H)}\xrightarrow[\epsilon\to 0^+]{}0.
\end{align*}
On the other hand, since $ r \mapsto\int_ r ^T\F(t- r) e \dot\varphi(t)\,\de t$ belongs to $L^\infty(0,T;H)$, we can write
\begin{align*}
&\int_0^T\int_0^t(\F(t- r) (e u^{\epsilon}(r)-eu^0), e \dot\varphi(t))_{H}\,\de r \,\de t
=\int_0^T(e u^{\epsilon}(r)-eu^0,\int_ r ^T\F(t- r) e \dot\varphi(t)\,\de t)_{H}\,\de r \\
&\xrightarrow[\epsilon\to 0^+]{} \int_0^T(e u^* (r)-eu^0,\int_ r ^T\F(t- r) e \dot\varphi(t)\,\de t)_{H}\,\de r=\int_0^T\int_0^t(\F(t- r) (e u^*(r)-eu^0), e \dot\varphi(t))_{H}\,\de r \,\de t.
\end{align*}
As a consequence, $u^*$ is a generalized solution to system~\eqref{eq:Fsystem}.
\end{proof}

We conclude this section by showing that for the fractional Kelvin-Voigt model, the generalized solution $u^*\in \mathcal C_w$ to~\eqref{eq:Fsystem} found before satisfies an energy-dissipation inequality. As before, for $t\in(0,T]$ we define the functions $\mathcal{E}^*(t)$ and $\mathcal{D}^*(t)$ as
\begin{align*}
&\mathcal{E}^*(t):=\frac{1}{2}\|\dot u^*(t)\|_H^2+\frac{1}{2}(\C e u^*(t),e u^*(t))_{H}\,\de t+\frac{1}{2}(\F(t)(e u^*(t)-eu^0),e u^*(t)-eu^0)_{H}\nonumber\\
&\hspace{5.5cm}-\frac{1}{2}\int_0^t(\dot \F(t-r)(eu^*(t)-e u^*(r)),eu^*(t)-e u^*(r))_{H}\,\de r,\\
&\mathcal D^*(t):=-\frac{1}{2}\int_0^t(\dot \F(r)(e u^*(r)-eu^0),e u^*(r)-eu^0)_{H}\,\de r\nonumber\\
&\hspace{5.5cm}+\frac{1}{2}\int_0^t\int_0^r(\ddot \F(r-s)(e u^*(r)-e u^*(s)),e u^*(r)-e u^*(s))_{H}\,\de s\,\de r.
\end{align*}
Notice that the integrals in $\mathcal E^*$ and $\mathcal D^*$ are well-posed, eventually with values $\infty$. Furthermore, we define the total work $\mathcal W_{tot}^*(t)$ for $t\in[0,T]$ as
\begin{align}\label{eq:totwork2}
\mathcal W^*_{tot}(t):&=\int_0^t [(f(r),\dot u^*(r)-\dot z(r))_H-(\dot N(r),u^*(r)-z(r))_{H_N}-(\dot u^*(r),\ddot z(r))_H+(\C eu^*(t),e\dot z(t))_H]\,\de r\nonumber\\
&\quad +(N(t),u^*(t)-z(t))_{H_N}-(N(0),u^0-z(0))_{H_N}+(\dot{u}^*(t),\dot{z}(t))_H -(u^1,\dot{z}(0))_H\nonumber\\
&\quad +\int_0^t(\F(t- r) (e u^*(r)-eu^0) , e \dot z(t))_{H}\,\de r-\int_0^t\int_0^ r(\F(r -s) (e u^*(s)-eu^0), e \ddot z(r))_{H}\,\de s\,\de r.
 \end{align}
We point out the total work $\mathcal W_{tot}^*$ is continuous in $[0,T]$ and that the definition given in~\eqref{eq:totwork2} is coherent with the one of~\eqref{tot-work} thanks to identity~\eqref{eq:totwork}. 

\begin{theorem}\label{thm:irr_enin}
Assume~\eqref{eq:dataFz}--\eqref{eq:F4}. Then the generalized solution $u^*\in\mathcal C_w$ to system~\eqref{eq:Fsystem} of Theorem~\ref{thm:irr_exis} satisfies for every $t\in (0,T]$ the following energy-dissipation inequality
\begin{equation}\label{dis-in-*}
 \mathcal{E}^*(t)+\mathcal{D}^*(t)\leq \frac{1}{2}\|u^1\|^2_H+\frac{1}{2}(\C eu^0,eu^0)_H+\mathcal{W}_{tot}^*(t).
 \end{equation}
In particular, $\mathcal E^*(t)$ and $\mathcal D^*(t)$ are finite for every $t\in(0,T]$.
\end{theorem}

\begin{proof}
Let us fix $t\in(0,T]$. For every $\epsilon\in(0,\delta_0)$ let $\ue\in \mathcal C_w$ be the generalized solution to system~\eqref{eq:Gsystem} with $\G$ replaced by $\G^\epsilon$ given by Lemma~\ref{lem:ir-con}. Thanks to Proposition~\ref{prop:reg_enin} we know that the function $\ue$ satisfies the energy-dissipation inequality~\eqref{eq:enin} and we can rewrite the total work \eqref{tot-work} as in~\eqref{eq:totwork2} since $z\in W^{2,1}(0,T;U_0)$ (as suggested by formula~\eqref{eq:totwork}). The convergences~\eqref{ep-point-conv} of Lemma~\ref{lem:ir-con}, and the lower semicontinuous property of the maps $v\mapsto \|v\|^2_H$, $w\mapsto (\C w,w)_H$ (by~\eqref{eq:C}), and $w\mapsto (\F(t) w,w)_H$ (by~\eqref{eq:F2}), imply 
\begin{align}
 \|\dot u^*(t)\|^2_H&\leq \liminf_{\epsilon\to 0^+}\|\du(t)\|^2_H,\label{ep-cont-first}\\
 (\C eu^*(t),eu^*(t))_H&\leq \liminf_{\epsilon\to 0^+}(\C e\ue(t),e\ue(t))_H,\\
 (\F(t) (eu^*(t)-eu^0),eu^*(t)-eu^0)_H&\leq \liminf_{\epsilon\to 0^+}(\F(t) (e\ue(t)-eu^0),e\ue(t)-eu^0)_H.\label{en-lim-G}
\end{align}
Moreover, by~\eqref{eq:F1} we have
\begin{align*}
 |((\F(t)-\G^\epsilon(t))(e\ue(t)-eu^0),e\ue(t)-eu^0)_H|&\leq \|\F(t)-\G^{\epsilon}(t)\|_B\|e\ue(t)-eu^0\|_H^2\nonumber\\
 &\leq 4M^2\|\F(t)-\F(t+\epsilon)\|_B\xrightarrow[\epsilon\to 0^+]{}0,
\end{align*}
being $M$ independent of $\epsilon$. Hence~\eqref{en-lim-G} reads as
\begin{equation}
 (\F(t)(eu^*(t)-eu^0),eu^*(t)-eu^0)_H\leq \liminf_{\epsilon\to 0^+}(\G^{\epsilon}(t)(e\ue(t)-eu^0),e\ue(t)-eu^0)_H.
\end{equation}
Similarly, by~\eqref{eq:F1},~\eqref{eq:F3}, and~\eqref{ep-point-conv}, for every $r\in(0,t)$ we have
% \begin{align*}
% (-\dot\F(t-r)(eu^*(t)-eu^*(r)),eu^*(t)-eu^*(r))_H\le\liminf_{\epsilon\to 0^+}(-\dot\F(t-r)(e\ue(t)-e\ue(r)),e\ue(t)-e\ue(r))_H.
% \end{align*}
% Again thanks to~\eqref{eq:F1} we get
% \begin{equation*}
% \|\dot{\G}^{\epsilon}(t-r)-\dot\F(t-r)\|_B=\|\dot\F(t-r+\epsilon)-\dot\F(t-r)\|_B\xrightarrow[\epsilon\to 0^+]{} 0,
% \end{equation*}
% therefore we can argue as before to deduce
\begin{align*}
(-\dot\F(t-r)(eu^*(t)-eu^*(r))&,eu^*(t)-eu^*(r))_H\\
&\le\liminf_{\epsilon\to 0^+}(-\dot \G^{\epsilon}(t-r)(e\ue(t)-e\ue(r)),e\ue(t)-e\ue(r))_H.
\end{align*}
In particular, we can use Fatou's lemma to obtain
\begin{align*}
\int_0^t(-\dot\F(t-r)(eu^*(t)-eu^*(r))&,eu^*(t)-eu^*(r))_H\,\de r\nonumber\\
&\le\liminf_{\epsilon\to 0^+}\int_0^{t}(-\dot\F(t-r)(e\ue(t)-e\ue(r)),e\ue(t)-e\ue(r))_H\,\de r.
\end{align*}
By arguing in a similar way, we can derive 
\begin{equation*}
\int_0^t(-\dot\F(r)(eu^*(r)-eu^0),eu^*(r)-eu^0)_H\,\de r\le\liminf_{\epsilon\to 0^+}\int_0^{t}(-\dot \G^{\epsilon}(r)(e\ue(r)-eu^0),e\ue(r)-eu^0)_H\,\de r.
\end{equation*}
For the term involving $\ddot \F$, we argue as we already did for $\dot\F$ and by using two times Fatou's lemma we get
% Now, let us consider the last term in the left-hand side. We fix $r\in(0,t)$ and for every $s\in(0,r)$ we have 
% \begin{align*}
% (\ddot\F(r-s)(eu^*(r)-eu^*(s))&,eu^*(r)-eu^*(s))_H\\
% &\le\liminf_{\epsilon\to 0^+}(\ddot\F(r-s)(e\ue(r)-e\ue(s)),e\ue(r)-e\ue(s))_H.
% \end{align*}
% Since for a.e. $s\in (0,r)$ it holds
% \begin{equation*}
% \|\ddot{\G}^{\epsilon}(r-s)-\ddot\F(r-s)\|_B=\|\ddot\F(r-s+\epsilon)-\ddot\F(r-s)\|_B\xrightarrow[n\to\infty]{} 0,
% \end{equation*}
% we can write for a.e. $s\in(0,r)$ 
% \begin{align*}
% (\ddot\F(r-s)(eu^*(r)-eu^*(s))&,eu^*(r)-eu^*(s))_H\\
% &\le\liminf_{\epsilon\to 0^+}(\ddot{\G}^{\epsilon}(r-s)(e\ue(r)-e\ue(s)),e\ue(r)-e\ue(s))_H.
% \end{align*}
% By applying again Fatou's lemma we conclude
% \begin{align*}
% \int_0^r(\ddot\F(r-s)(eu^*(r)-eu^*(s))&,eu^*(r)-eu^*(s))_H\,\de s\\
% &\le\liminf_{\epsilon\to 0^+}\int_0^r(\ddot{\G}^{\epsilon}(r-s)(e\ue(r)-e\ue(s)),e\ue(r)-e\ue(s))_H\,\de s.
% \end{align*}
% Since this is true for every $r\in(0,t)$, arguing as before we obtain
\begin{align*}
\int_0^t\int_0^r(\ddot\F(r-s)(eu^*(r)-eu^*&(s)),eu^*(r)-eu^*(s))_H\,\de s\,\de r\\
&\le\liminf_{\epsilon\to 0^+}\int_{0}^{t}\int_0^{r}(\ddot{\G}^{\epsilon}(r-s)(e\ue(r)-e\ue(s)),e\ue(r)-e\ue(s))_H\,\de s\,\de r.
\end{align*}

It remains to study the right-hand side of~\eqref{eq:enin} with the formulation of the total work as in~\eqref{eq:totwork2}. Thanks to Lemma~\ref{lem:ir-con} and the fact that $\G^\epsilon\to \F$ in $L^1(0,T;B)$ we deduce
\begin{align}
\int_0^{t}(f(r),\du(r))_H\,\de r &\xrightarrow[\epsilon\to 0^+]{} \int_0^t(f(r),\dot{u}^*(r))_H\,\de r,\\
\int_0^{t}(\C e\ue(r),e\dot{z}(r))_H\,\de r &\xrightarrow[\epsilon\to 0^+]{} \int_0^t (\C eu^*(r),e\dot z(r))_H\,\de r ,\\
\int_0^{t}(\G^{\epsilon}(t-r)(e\ue(r)-eu^0),e\dot{z}(r))_H\,\de r &\xrightarrow[\epsilon\to 0^+]{} \int_0^t (\F(t-r) (eu^*(r)-eu^0),e\dot z(r))_H\,\de r ,\\
(\du(t),\dot z(t))_H-\int_0^{t}(\du(r),\ddot z(r))_H\,\de r &\xrightarrow[\epsilon\to 0^+]{} (\dot u^*(t),\dot z(t))_H - \int_0^{t}(\dot u^*(r),{\ddot z}(r))_H\,\de r ,\\
(N(t),\ue(t))_{H_N}-\int_0^{t}(N(r),\du(r))_{H_N}\,\de r &\xrightarrow[\epsilon\to 0^+]{}(N(t),u^*(t))_{H_N}- \int_0^{t}(\dot N(r),u^*(r))_{H_N}\,\de r.
\end{align}
It remains to study the term
\begin{equation*}
\int_0^t\int_0^ r(\G^{\epsilon}(r -s) (e \ue(s)-eu^0), e \ddot z(r))_{H}\,\de s\,\de r.
\end{equation*}
For a.e. $r\in(0,t)$ we have
\begin{align*}
&\int_0^ r (\G^\epsilon(r -s) (e \ue(s)-eu^0),e\ddot z(r))_H\,\de s\xrightarrow[\epsilon\to 0^+]{}\int_0^ r (\F(r -s) (e u^*(s)-eu^0),e\ddot z(r))_H\,\de s\\
&\left| \int_0^ r (\G^\epsilon(r -s) (e \ue(s)-eu^0),e\ddot z(r))_H\,\de s\right|\le 2M\|\F\|_{L^1(0,T+\delta_0;B)}\|e\ddot z(r)\|_H\in L^1(0,t),
\end{align*}
with $M$ independent of $\epsilon$. By the dominated convergence theorem we conclude
\begin{equation}\label{ep-cont-last}
\int_0^t\int_0^ r(\G^{\epsilon}(r -s) (e \ue(s)-eu^0), e \ddot z(r))_{H}\,\de s\,\de r\xrightarrow[\epsilon\to 0^+]{}\int_0^t\int_0^ r(\F(r -s) (e u^*(s)-eu^0), e \ddot z(r))_{H}\,\de s\,\de r.
\end{equation}
By combining~\eqref{ep-cont-first}--\eqref{ep-cont-last} we deduce the energy-dissipation inequality~\eqref{dis-in-*} for every $t\in(0,T]$.
\end{proof}

\begin{remark}
Although we do not have any information about $L^1$-integrability of $\dot \F$ and $\ddot \F$ in $t=0$, for the generalized solution $u^*$ of Theorem~\ref{thm:irr_exis} we obtain that the energy terms $\mathcal E^*$ and $\mathcal D^*$ are finite. 
\end{remark}

\begin{corollary}
Assume~\eqref{eq:dataFz}--\eqref{eq:F4}. Then the generalized solution $u^*\in\mathcal C_w$ to system~\eqref{eq:Fsystem} of Theorem~\ref{thm:irr_exis} satisfies
\begin{equation}\label{lim-en-*}
 \lim_{t\to 0^+}\mathcal E^*(t)=\frac{1}{2}\|u^1\|_H^2+\frac{1}{2}(\C eu^0,eu^0)_H.
\end{equation}
In particular,~\eqref{dis-in-*} holds true also in $t=0$ and
\begin{equation*}
\lim_{t\to 0^+}\|u^*(t)-u^0\|_{U_T}=0,\quad\lim_{t\to 0^+}\|\dot u^*(t)-u^1\|_H=0.
\end{equation*}
\end{corollary}

\begin{proof}
By~\eqref{dis-in-*} for every $t\in(0,T]$ we have
\begin{equation*}
 \frac{1}{2}\|\dot u^*(t)\|_H^2+\frac{1}{2}(\C eu^0,eu^0)_H\le \mathcal E^*(t)\le \frac{1}{2}\|u^1\|_H^2+\frac{1}{2}(\C eu^0,eu^0)_H+\mathcal W_{tot}^*(t).
\end{equation*}
Since $u^*\in C_w^0([0,T];U_T)$ and $\dot u^*\in C_w^0([0,T];H)$ we get
\begin{equation*}
\frac{1}{2}\|u^1\|_H^2+\frac{1}{2}(\C eu^0,eu^0)_H\le \liminf_{t\to 0^+}\mathcal E^*(t)\le\limsup_{t\to 0^+}\mathcal E^*(t)\le \frac{1}{2}\|u^1\|_H^2+\frac{1}{2}(\C eu^0,eu^0)_H.
\end{equation*}
Therefore, we get~\eqref{lim-en-*}. As consequence of this, we derive
\begin{equation*}
 \lim_{t\to 0^+}\|\dot u^*(t)\|_H^2=\|u^1\|_H^2,\quad\lim_{t\to 0^+} (\C eu^*(t),eu^*(t))_H=(\C eu^0,eu^0)_H,
\end{equation*}
and this conclude the proof.
\end{proof}

\section{Uniqueness for a not moving crack}\label{sec:uniq}

Let us consider the case of a domain with a fixed crack, i.e. $\Gamma_T=\Gamma_0$ (possibly $\Gamma_T=\emptyset$). In this case we can show that the generalized solution to~\eqref{eq:Fsystem} is unique. Similar results can be found in literature in~\cite{CaCaVa,Opa-Su}, but they are proved for slightly different models.

The proof of the uniqueness is based on a particular energy estimate which holds for the primitive of a generalized solution. To this aim, we need to estimate
\begin{equation*}
\int_0^t\int_0^r(\F(r-s)eu(s),eu(r))_{H}\,\de s\,\de r
\end{equation*}
and we start with the following identity which is true for a regular tensor $\K$.

\begin{lemma}
Let $\K\in C^1([0,T];B)$ and $v\in L^2(0,T;U_0)$. Then, for every $t\in[0,T]$
\begin{align}\label{eq:Kidentity}
&\int_0^t(\frac{\de}{\de r}\int_0^r\K(r-s)e v(s)\,\de s ,ev(r))_{H}\,\de r=\frac{1}{2}\int_0^t(\K(t-r)e v(r),e v(r))_{H}\,\de r\nonumber\\
&\hspace{2cm}+\frac{1}{2}\int_0^t(\K(r)ev(r),ev(r))_{H}\,\de r-\frac{1}{2}\int_0^t\int_0^r(\dot \K(r-s)(ev(r)-ev(s)),ev(r)-ev(s))_{H}\,\de s\,\de r.
\end{align}
\end{lemma}

\begin{proof}
Let us fix $t\in [0,T]$ and let us analyze the right hand-side of~\eqref{eq:Kidentity}. We have
\begin{align}\label{K-1}
-\frac{1}{2}\int_0^t\int_0^r(\dot \K(r&-s)(ev(r)-ev(s)),ev(r)-ev(s))_{H}\,\de s\,\de r=\int_0^t\int_0^r(\dot \K(r-s)ev(s),ev(r))_{H}\,\de s\,\de r\nonumber\\
&-\frac{1}{2}\int_0^t\int_0^r(\dot \K(r-s)ev(s),ev(s))_{H}\,\de s\,\de r-\frac{1}{2}\int_0^t\int_0^r(\dot \K(r-s)ev(r),ev(r))_{H}\,\de s\,\de r.
\end{align}
Notice that
\begin{align}
-\frac{1}{2}\int_0^t\int_0^r(\dot \K(r-s)ev(r),ev(r))_{H}\,\de s\,\de r&=-\frac{1}{2}\int_0^t(\left(\int_0^r \dot \K(r-s)\de s \right)ev(r),ev(r))_{H}\,\de s\,\de r\nonumber\\
&=-\frac{1}{2}\int_0^t( \K(r)ev(r),ev(r))_{H}\,\de r+\frac{1}{2}\int_0^t( \K(0)ev(r),ev(r))_{H}\,\de r,
\end{align}
and that for a.e. $r\in (0,t)$
\begin{equation*}
\frac{\de}{\de r}\int_0^r(\K (r-s)ev(s),ev(s))_H\,\de s=( \K(0)ev(r),ev(r))_{H}+\int_0^r( \dot \K(r-s)ev(s),ev(s))_{H}\,\de s,
\end{equation*}
from which we deduce
\begin{align}\label{K-3}
-\frac{1}{2}\int_0^t(\K (t-r)ev(r),ev(r))_H\,\de r&=-\frac{1}{2}\int_0^t \frac{\de}{\de r}\int_0^r(\K (r-s)ev(s),ev(s))_H\,\de s\,\de r\nonumber\\
&= -\frac{1}{2}\int_0^t( \K(0)ev(r),ev(r))_{H}\,\de r -\frac{1}{2}\int_0^t\int_0^r( \dot \K(r-s)ev(s),ev(s))_{H}\,\de s\,\de r.
\end{align}
By~\eqref{K-1}--\eqref{K-3} we can say
\begin{align*}
-\frac{1}{2}\int_0^t\int_0^r(\dot \K(r-s)(ev(r)&-ev(s)),ev(r)-ev(s))_{H}\,\de s\,\de r\nonumber\\
&=\int_0^t\int_0^r(\dot \K(r-s)ev(s),ev(r))_{H}\,\de s\,\de r+\int_0^t( \K(0)ev(r),ev(r))_{H}\,\de r\nonumber\\
&-\frac{1}{2}\int_0^t( \K(r)ev(r),ev(r))_{H}\,\de r-\frac{1}{2}\int_0^t(\K (t-r)ev(r),ev(r))_H\,\de r,
\end{align*}
and thanks to the following relation
\begin{equation*}
\frac{\de}{\de r}\int_0^r\K(r-s)e v(s)\,\de s =\K(0)ev(r)+\int_0^r\dot \K(r-s)ev(s)\de s\quad\text{for a.e. $r\in(0,t)$},
\end{equation*}
we can conclude the proof.
\end{proof}

\begin{lemma}
Let $\F$ be satisfying~\eqref{eq:F1}--\eqref{eq:F4} and $u\in C_w^0([0,T];U_0)$. Then for every $t\in[0,T]$ it holds
\begin{equation}\label{eq:Glowerest}
\int_0^t\int_0^r(\F(r-s)eu(s),eu(r))_{H}\,\de s\,\de r\ge 0.
\end{equation}
\end{lemma}

\begin{proof}
First, we fix $\epsilon\in (0,\delta_0)$ and we consider for every $t\in[0,T]$ the following regularized kernel
\begin{equation*}
\G^\epsilon(t):=\F(t+\epsilon).
\end{equation*}
Moreover, we fix $t\in[0,T]$ and we define for every $r\in[0,t]$ a primitive of $u$ in the following way
\begin{equation*}
v(r):=-\int_r^tu(s)\,\de s.
\end{equation*}
Clearly $\G^\epsilon\in C^2([0,T];B)$ and after an integration by parts, since $ev(t)=0$, we obtain
\begin{align*}
\int_0^t\int_0^r(\G^\epsilon(r-s)eu(s),eu(r))_{H}\,\de s\,\de r&=\int_0^t\int_0^r(\G^\epsilon(r-s)eu(s),e\dot v(r))_{H}\,\de s\,\de r\\
&=-\int_0^t(\G^\epsilon(0)e\dot v(r),ev(r))_{H}\,\de r-\int_0^t\int_0^r(\dot \G^\epsilon(r-s)eu(s),ev(r))_{H}\,\de s\,\de r\\
&=\frac{1}{2}(\G^\epsilon(0)e v(0),ev(0))_{H}-\int_0^t\int_0^r(\dot \G^\epsilon(r-s)eu(s),ev(r))_{H}\,\de s\,\de r.
\end{align*}
Moreover, we have
\begin{equation*}
\int_0^r\dot \G^\epsilon(r-s)eu(s)\,\de s=\frac{\de}{\de r}\int_0^r\dot\G^\epsilon(r-s)ev(s)\,\de s-\dot \G^\epsilon(r)ev(0).
\end{equation*}
Therefore, by~\eqref{eq:Kidentity} we can write
\begin{align*}
\int_0^t\int_0^r(\dot \G^\epsilon(r-s)eu(s),ev(r))_{H}\,\de s\,\de r&=\int_0^t(\frac{\de}{\de r}\int_0^r\dot\G^\epsilon(r-s)ev(s)\,\de s-\dot \G^\epsilon(r)ev(0),ev(r))_{H}\,\de r\\
&=\frac{1}{2}\int_0^t(\dot \G^\epsilon(t-r)e v(r),e v(r))_{H}\,\de r+\frac{1}{2}\int_0^t(\dot \G^\epsilon(r)ev(r),ev(r))_{H}\,\de r\\
&\quad-\frac{1}{2}\int_0^t\int_0^r\ddot \G^\epsilon(r-s)(ev(r)-ev(s)),ev(r)-ev(s))_{H}\,\de s\,\de r\\
&\quad-\int_0^t(\dot \G^\epsilon(r)ev(0),ev(r))_{H}\,\de r,
\end{align*}
which implies 
\begin{align*}
\int_0^t\int_0^r(\G^\epsilon(r-s)eu(s),eu(r))_{H}\,\de s\,\de r&=\frac{1}{2}(\G^\epsilon(0)e v(0),ev(0))_{H}+\int_0^t(\dot \G^\epsilon(r)ev(0),ev(r))_{H}\,\de r\\
&\quad-\frac{1}{2}\int_0^t(\dot \G^\epsilon(t-r)e v(r),e v(r))_{H}\,\de r-\frac{1}{2}\int_0^t(\dot \G^\epsilon(r)ev(r),ev(r))_{H}\,\de r\\
&\quad+\frac{1}{2}\int_0^t\int_0^r(\ddot \G^\epsilon(r-s)(ev(r)-ev(s)),ev(r)-ev(s))_{H}\,\de s\,\de r\\
&\ge \frac{1}{2}(\G^\epsilon(0)e v(0),ev(0))_{H}+\frac{1}{2}\int_0^t(\dot \G^\epsilon(r)ev(0),ev(0))_{H}\,\de r\\
&\quad-\frac{1}{2}\int_0^t(\dot \G^\epsilon(t-r)e v(r),e v(r))_{H}\,\de r\\
&\quad+\frac{1}{2}\int_0^t\int_0^r(\ddot \G^\epsilon(r-s)(ev(r)-ev(s)),ev(r)-ev(s))_{H}\,\de s\,\de r\\
&=\frac{1}{2}(\G^\epsilon(t)e v(0),ev(0))_{H}-\frac{1}{2}\int_0^t(\dot \G^\epsilon(t-r)e v(r),e v(r))_{H}\,\de r\\
&\quad+\frac{1}{2}\int_0^t\int_0^r\ddot \G^\epsilon(r-s)(ev(r)-ev(s)),ev(r)-ev(s))_{H}\,\de s\,\de r\ge 0.
\end{align*}
By sending $\epsilon\to 0^+$ we conclude.
\end{proof}

We can now state our uniqueness result.

\begin{theorem}
Assume~\eqref{eq:dataFz}--\eqref{eq:F4} and $\Gamma_T=\Gamma_0$. Then there exists at most one generalized solution to system~\eqref{eq:Fsystem}.
\end{theorem}

\begin{proof}
Let $u_1,u_2\in \mathcal C_w$ be two generalized solutions to~\eqref{eq:Fsystem}. Then $u:=u_1-u_2$ satisfies equality~\eqref{eq:Fgen} with $z=N=f=u^0=u^1=0$.
Consider the function $\beta\colon [0,T]\to (U_0^D)'$ defined for every $r\in[0,T]$ as
\begin{equation*}
\langle \beta(r),v\rangle_{(U^D_0)'}:=(\dot u(r),v)_H+\int_0^r(\mathbb Ce u(s),e v)_{H}\,\de s+\int_0^r(\F(r-s)e u(s),e v)_{H}\,\de s
\end{equation*}
for every $v\in U^D_0$. Clearly $\beta\in C_w^0([0,T];(U^D_0)')$, $\beta(0)=0$ since $\dot u(0)=0$ in $(U^D_0)'$, and by~\eqref{eq:Fgen} we derive
\begin{equation*}
\int_0^T\langle \beta(r),v\rangle_{(U^D_0)'}\dot\psi(r)\,\de r=0\quad\text{for every $v\in U^D_0$ and $\psi\in C_c^1(0,T)$}. 
\end{equation*}
Therefore $\beta$ is constant in $[0,T]$, which gives $\beta(t)=0$ in $(U^D_0)'$ for every $t\in[0,T]$, namely
\begin{equation*}
(\dot u(r),v)_H+\int_0^r(\mathbb Ce u(s),e v)_{H}\,\de s+\int_0^r(\F(r-s)e u(s),e v)_{H}\,\de s=0\quad\text{for every $v\in U^D_0$ and $r\in[0,T]$}.
\end{equation*}
In particular, for every $t\in[0,T]$ we deduce
\begin{equation*}
\int_0^t(\dot u(r),u(r))_H\,\de r+\int_0^t\int_0^r(\mathbb Ce u(s),e u(r))_{H}\,\de s\,\de r+\int_0^t\int_0^r(\F(r-s)e u(s),e u(r))_{H}\,\de s\,\de r=0.
\end{equation*}
Hence, by~\eqref{eq:Glowerest} we conclude
\begin{equation*}
\frac{1}{2}\|u(t)\|_H^2+\frac{1}{2}( \C\left(\int_0^t eu(r)\,\de r\right),\int_0^t eu(r)\,\de r)_H\le 0\quad\text{for every $t\in[0,T]$}.
\end{equation*}
Therefore, since both terms are non-negative, we get that $u(t)=0$ for every $t\in[0,T]$.
\end{proof}

\begin{acknowledgements}
The authors wish to thank Professors Gianni Dal Maso for the many useful discussions on the topic. 
The authors are members of the {\em Gruppo Nazionale per l'Analisi Ma\-te\-ma\-ti\-ca, la Probabilit\`a e le loro Applicazioni} (GNAMPA) of the {\em Istituto Nazionale di Alta Matematica} (INdAM).
\end{acknowledgements}

\begin{declaration}
None.
\end{declaration}

%---------------------------
% Bibliography
%---------------------------


\begin{thebibliography}{99}

\bibitem{At-Opa}{\sc T.M. Atanackovic, M. Janev, L. Oparnica, S. Pilipovic, and D. Zorica}: Space-time fractional Zener wave equation. {\it Proc. A.} {\bf 471} (2015), 25 pp.

\bibitem{C2}{\sc M.~Caponi}: Linear hyperbolic systems in domains with growing cracks. {\it Milan J. Math.} {\bf 85} (2017), 149--185.

\bibitem{Sap-ca} {\sc M. Caponi and F. Sapio}: A dynamic model for viscoelastic materials with prescribed growing cracks. {\it Ann. Mat. Pura Appl.}, {\bf 199} (2020), 1263--1292.

\bibitem{Car-Co}{\sc A. Carbotti and G. Comi}: A note on Riemann-Liouville fractional Sobolev spaces. To appear on {\it  Commun. Pure Appl. Anal} (2020). Preprint arXiv:2003.09515.

\bibitem{CaCaVa}{\sc S. Carillo, V. Valente, G. Vergara Caffarelli}: A linear viscoelasticity problem with a singular memory kernel: an existence and uniqueness result. {\it Differential Integral Equations} {\bf 26} (2013), 1115--1125. 

\bibitem{DM-Lar}{\sc G.~Dal Maso and C.J.~Larsen}: Existence for wave equations on domains with arbitrary growing cracks. {\it Atti Accad. Naz. Lincei Rend. Lincei Mat. Appl.} {\bf 22} (2011), 387--408.

\bibitem{DM-Luc}{\sc G.~Dal Maso and I.~Lucardesi}: The wave equation on domains with cracks growing on a prescribed path: existence, uniqueness, and continuous dependence on the data. {\it Appl. Math. Res. Express. AMRX} (2017), 184--241.

\bibitem{DL}{\sc R. Dautray and J.L. Lions}: Analyse math\'ematique et calcul num\'erique pour les sciences et les techniques. Vol. 8. \'Evolution: semi-groupe, variationnel. Masson, Paris, 1988.

\bibitem{Val-DiP-Ve}{\sc S. Dipierro, E. Valdinoci, and V. Vespri}: Decay estimates for evolutionary equations with fractional time-diffusion. {\it J. Evol. Equ.} {\bf 19}, (2019) 435--462.

\bibitem{Dra}{\sc C.S. Drapaca and S. Sivaloganathan}:
A fractional model of continuum mechanics.
{\it J. Elasticity} {\bf 107} (2012), 105–123.

\bibitem{Fra}{\sc M. Fabrizio}: Fractional rheological models for thermomechanical systems. Dissipation and free energies. {\it Fract. Calc. Appl. Anal.} {\bf 17} (2014), 206--223.

%\bibitem{FaBaEs}{\sc E. Farno, J.C. Baudez, and N. Eshtiaghi}: Comparison between classical Kelvin-Voigt and fractional derivative Kelvin-Voigt models in prediction of linear viscoelastic behaviour of waste activated sludge. {\it Sci. Total Environ.} {\bf 613--614}, (2018) 1031--1036.

\bibitem{Kilbas}{\sc A.A. Kilbas, H.M. Srivastava, and J.J. Trujillo}: Theory and applications of fractional differential equations. North-Holland Mathematics Studies, Amsterdam, 2006.

\bibitem{Opa}{\sc S. Konjik, L. Oparnica, and D. Zorica}: Waves in fractional Zener type viscoelastic media. {\it J. Math. Anal. Appl.} {\bf 365}, (2010) 259--268.

\bibitem{Li-Liu}{\sc L. Li and J.G. Liu}: A generalized definition of Caputo derivatives and its application to fractional ODEs. {\it SIAM J. Math. Anal.} {\bf 50}, (2018) 2867--2900.

\bibitem{Mai}{\sc F. Mainardi}: Fractional calculus and waves in linear viscoelasticity. Imperial College Press, London, 2010.

\bibitem{NS}{\sc S. Nicaise and A.M. S\"andig}: Dynamic crack propagation in a 2D elastic body: the out--of--plane case. {\it J. Math. Anal. Appl.} {\bf 329} (2007), 1--30.

\bibitem{OSY}{\sc O.A. Oleinik, A.S. Shamaev, and G.A. Yosifian}: Mathematical problems in elasticity and homogenization. North-Holland Publishing Co., Amsterdam, 1992.

\bibitem{Opa-Su}{\sc L. Oparnica and E. Süli}: Well-posedness of the fractional Zener wave equation for heterogeneous viscoelastic materials. {\it Fract. Calc. Appl. Anal.} {\bf 23}, (2020) 126--166.

\bibitem{Po}{\sc I. Podlubny} Fractional Differential Equations.  Academic Press, San Diego, 1999.

\bibitem{SaKiMa} {\sc S.G. Samko, A.A. Kilbas, and O.I. Marichev}: Fractional integrals and derivatives. Theory and applications. Edited and with a foreword by S. M. Nikol'ski\u\i. Translated from the 1987 Russian original. Revised by the authors. Gordon and Breach Science Publishers, Yverdon, 1993. 

\bibitem{Sap} {\sc F. Sapio}: A dynamic model for viscoelasticity in domains with time-dependent cracks. Submitted for publication (2020). Preprint SISSA 14/2020/MATE.

\bibitem{Sch-Met} {\sc H. Schiessel, R. Metzler, A. Blumen, and T.F. Nonnenmacher}: Generalized viscoelastic models: their fractional equations with solutions. {\it J. Phys. A: Math. Gen.} {\bf 28} (1995), 6567--6584.

\bibitem{T1}{\sc E. Tasso}: Weak formulation of elastodynamics in domains with growing cracks. {\it Ann. Mat. Pura Appl.}, {\bf 199} (2020), 1571--1595.

\bibitem{Zhu}{\sc H.H. Zhu, L.C. Liu, H.F. Pei, and B. Shi}: Settlement analysis of viscoelastic foundation under vertical line load using a fractional Kelvin-Voigt model. {\it Geomech. Eng.}, {\bf 4} (2012), 67--68.

\end{thebibliography}
\end{document}